\documentclass[11pt]{article}
\usepackage{latexsym,amsfonts,amssymb,amsmath,amsthm}
\usepackage{graphicx}
\usepackage{cite}
\usepackage[usenames,dvipsnames]{color}
\usepackage{ulem}
\usepackage[colorlinks=true]{hyperref}
\hypersetup{urlcolor=blue, citecolor=red}

\newcommand{\norm}[1]{||#1||}

\begin{document}
\setlength{\baselineskip}{16pt}

\parindent 0.5cm
\evensidemargin 0cm \oddsidemargin 0cm \topmargin 0cm \textheight
22cm \textwidth 16cm \footskip 2cm \headsep 0cm

\newtheorem{theorem}{Theorem}[section]
\newtheorem{lemma}[theorem]{Lemma}
\newtheorem{proposition}[theorem]{Proposition}
\newtheorem{definition}{Definition}[section]
\newtheorem{example}{Example}[section]
\newtheorem{corollary}[theorem]{Corollary}

\newtheorem{remark}{Remark}[section]
\newtheorem{property}[theorem]{Property}
\numberwithin{equation}{section}
\newtheorem{mainthm}{Theorem}
\newtheorem{mainlem}{Lemma}

\numberwithin{equation}{section}

\def\p{\partial}
\def\I{\textit}
\def\R{\mathbb R}
\def\C{\mathbb C}
\def\u{\underline}
\def\l{\lambda}
\def\a{\alpha}
\def\O{\Omega}
\def\e{\epsilon}
\def\ls{\lambda^*}
\def\D{\displaystyle}
\def\wyx{ \frac{w(y,t)}{w(x,t)}}
\def\imp{\Rightarrow}
\def\tE{\tilde E}
\def\tX{\tilde X}
\def\tH{\tilde H}
\def\tu{\tilde u}
\def\d{\mathcal D}
\def\aa{\mathcal A}
\def\DH{\mathcal D(\tH)}
\def\bE{\bar E}
\def\bH{\bar H}
\def\M{\mathcal M}
\renewcommand{\labelenumi}{(\arabic{enumi})}

\def\disp{\displaystyle}
\def\undertex#1{$\underline{\hbox{#1}}$}
\def\card{\mathop{\hbox{card}}}
\def\sgn{\mathop{\hbox{sgn}}}
\def\exp{\mathop{\hbox{exp}}}
\def\OFP{(\Omega,{\cal F},\PP)}
\newcommand\JM{Mierczy\'nski}
\newcommand\RR{\ensuremath{\mathbb{R}}}
\newcommand\CC{\ensuremath{\mathbb{C}}}
\newcommand\QQ{\ensuremath{\mathbb{Q}}}
\newcommand\ZZ{\ensuremath{\mathbb{Z}}}
\newcommand\NN{\ensuremath{\mathbb{N}}}
\newcommand\PP{\ensuremath{\mathbb{P}}}
\newcommand\abs[1]{\ensuremath{\lvert#1\rvert}}

\newcommand\normf[1]{\ensuremath{\lVert#1\rVert_{f}}}
\newcommand\normfRb[1]{\ensuremath{\lVert#1\rVert_{f,R_b}}}
\newcommand\normfRbone[1]{\ensuremath{\lVert#1\rVert_{f, R_{b_1}}}}
\newcommand\normfRbtwo[1]{\ensuremath{\lVert#1\rVert_{f,R_{b_2}}}}
\newcommand\normtwo[1]{\ensuremath{\lVert#1\rVert_{2}}}
\newcommand\norminfty[1]{\ensuremath{\lVert#1\rVert_{\infty}}}

\title{Asymptotic behavior of semilinear parabolic equations on the circle with time almost-periodic/recurrent dependence}

\author {
\\
Wenxian Shen\\
Department of Mathematics and Statistics\\
 Auburn University, Auburn, AL 36849, USA
\\
\\
Yi Wang\thanks{Partially supported by NSF of China No.11371338, 11471305, Wu Wen-Tsun Key Laboratory and the Fundamental
Research Funds for the Central Universities.}  $\,\,$ and $\,$ Dun Zhou\thanks{Partially supported by NSF of China No.11601498, China Postdoctoral Science Foundation No. 2016M600480 and Wu Wen-Tsun Key Laboratory.}\\
School of Mathematical Science\\
 University of Science and Technology of China
\\ Hefei, Anhui, 230026, P. R. China
\\
\\
}
\date{}

\maketitle

\begin{abstract}
We study topological structure of the $\omega$-limit sets of the skew-product semiflow generated by the following scalar reaction-diffusion equation
\begin{equation*}
u_{t}=u_{xx}+f(t,u,u_{x}),\,\,t>0,\,x\in S^{1}=\mathbb{R}/2\pi \mathbb{Z},
\end{equation*}
where $f(t,u,u_x)$ is $C^2$-admissible with time-recurrent structure including almost-periodicity and almost-automorphy. Contrary to the time-periodic cases (for which any $\omega$-limit set can be imbedded into a periodically forced circle flow), it is shown that one cannot expect that any $\omega$-limit set can be imbedded into an almost-periodically forced circle flow even if $f$ is uniformly almost-periodic in $t$.

More precisely, we prove that, for a given $\omega$-limit set $\Omega$, if ${\rm dim}V^c(\O)\le 1$ ($V^c(\O)$ is the center space associated with $\O$), then $\Omega$ is either spatially-homogeneous or spatially-inhomogeneous; and moreover, any spatially-inhomogeneous $\O$ can be imbedded into a time-recurrently forced circle flow (resp. imbedded into an almost periodically-forced circle flow if $f$ is uniformly almost-periodic in $t$). On the other hand, when ${\rm dim}V^c(\O)>1$, it is pointed out that the above embedding property cannot hold anymore. Furthermore, we also show the new phenomena of the residual imbedding into a time-recurrently forced circle flow (resp. into an almost automorphically-forced circle flow if $f$ is uniformly almost-periodic in $t$) provided that $\dim V^c(\Omega)=2$ and $\dim V^u(\Omega)$ is odd. All these results reveal that for such system there are essential differences between time-periodic cases and non-periodic cases.
\end{abstract}

\section{Introduction}

In this paper we consider the following scalar reaction-diffusion equation on the circle
$S^{1}=\mathbb{R}/2\pi \mathbb{Z}$:
\begin{equation}\label{equation-1}
u_{t}=u_{xx}+f(t,u,u_{x}),\,\,t>0,\,x\in S^{1},
\end{equation}
 where $f=f(t,u,u_x)$ is $C^2$-admissible and time-recurrent in $t$ including time-periodic, almost periodic and almost automorphic phenomena as special cases (see Definition \ref{admissible}) .

There are already many works concerning with the long time behavior of bounded solutions of \eqref{equation-1} in autonomous or time-periodic cases (see, e.g. \cite{Chen1989160,CRo,JR1,JR2,Massatt1986,Matano,SF1}). However, in practical problems, large quantities of systems evolve influenced by external effects which are roughly but not exactly periodic, or under environmental forcing which exhibits different, non-commensurate periods. Thus, using  quasi-periodic or almost periodic equations, or even certain nonautonomous
equations to characterized models with such time dependence are more appropriate. Based on these, we are trying to portray
the long time behavior of bounded solutions  of \eqref{equation-1} with time-recurrent structures including almost periodicity and almost automorphy, which  boils down to the problem of understanding the structure of  $\omega$-limit sets of
the skew-product semiflow generated by \eqref{equation-1}.

To be more precise, let $f(t,u,p)\in C(\mathbb{R}\times \mathbb{R} \times \mathbb{R},\mathbb{R})$ be a $C^2$-admissible function. Then $f_{\tau}(t,u,p)=f(t+\tau,u,p)(\tau \in \RR)$ generates a family
$\{f_{\tau}|\tau \in \mathbb{R}\}$ in the space of continuous functions $C(\mathbb{R}\times \mathbb{R} \times \mathbb{R},\mathbb{R})$ equipped with the compact open topology. The closure $H(f)$ of $\{f_{\tau}|\tau\in \mathbb{R}\}$ in
the compact open topology, called the hull of $f$, is a compact metric space and every $g\in H(f)$ has the same regularity as $f$. Thus, the time-translation $g\cdot t\equiv g_{t}\,(g\in H(f))$ defines a compact flow on $H(f)$.  We further assume that $f$ is {\it time-recurrent} or, in other words, the flow on $H(f)$ is minimal. This means that $H(f)$ is a minimal set of the flow, that is, it is the only nonempty compact subset of itself that is invariant under the flow $g\cdot t$. This is true, for instance, when $f$ is a uniformly almost periodic or, more generally, a uniformly almost automorphic function (Definition \ref{almost}).

Equation \eqref{equation-1} naturally
induces a family of equations associated to each $g\in H(f)$,
\begin{equation}\label{equation-lim1}
u_{t}=u_{xx}+g(t,u,u_{x}),\,\,\quad t> 0,\quad x\in S^{1}.
\end{equation}
To understand the long time behavior of bounded solutions of \eqref{equation-1}, we study the long time behavior of bounded solutions of \eqref{equation-lim1} for any $g\in H(f)$.
Assume that $X$ is the fractional power space associated with the operator $u\rightarrow
-u_{xx}:H^{2}(S^{1})\rightarrow L^{2}(S^{1})$ satisfies $X\hookrightarrow C^{1}(S^{1})$ (that is, $X$ is compact embedded in $C^{1}(S^{1})$). For any $u\in X$, \eqref{equation-lim1} defines (locally) a unique solution $\varphi(t,\cdot;u,g)$ in $X$ with $\varphi(0,\cdot;u,g)=u(\cdot)$ and it continuously depends on $g\in H(f)$ and $u\in
X$. Consequently, \eqref{equation-lim1} admits a (local) skew-product semiflow $\Pi^{t}$ on $X\times
H(f)$:
\begin{equation}\label{equation-lim2}
\Pi^{t}(u,g)=(\varphi(t,\cdot;u,g),g\cdot t),\quad t\ge 0.
\end{equation}
It follows from \cite{Hen} (see also \cite{Hess,Mierczynski}) and the standard a priori estimates for parabolic equations, if
$\varphi(t,\cdot;u,g) (u\in X)$ is bounded in $X$ in the existence interval of the solution, then $u$ is a globally defined classical solution. Note that, for any $\delta>0$, $\{\varphi(t,\cdot;u,g):t\ge \delta\}$ is relatively compact in $X$. Consequently, the $\omega$-limit set $\omega(u,g)$ of the bounded semi-orbit $\Pi^t(u,g)$ in $X\times H(f)$ is a nonempty connected compact subset of $X\times H(f)$.
The study of the long time behavior of the bounded solution $\varphi(t,\cdot;u,g)$ of \eqref{equation-lim1}  then boils down to the problem of understanding the structure of the $\omega$-limit set $\omega(u,g)$.

For the autonomous case or, equivalently, if $H(f)=\{f\}$, it is already known that any $\omega$-limit set $\omega(u)$ can be embedded into $\mathbb{R}^2$ (cf. the Poincar\'{e}-Bendixson type Theorem by Fiedler and Mallet-Paret \cite{Fiedler}; see also in \cite{Fiedler1}); and moreover, for \eqref{equation-1}, $\omega(u)$ is either a rotating wave, or contained in a set of equilibria differing only by phase shift in $x$ (see Massatt \cite{Massatt1986} or Matano \cite{Matano}).
In the case that $f$ is time-periodic with period $1$  (or equivalently, $H(f)$ is homeomorphic to the circle $\mathcal{T}^1=\mathbb{R}/\mathbb{Z}$), one may typically track the asymptotic behavior of bounded solutions by considering the $\omega$-limit set $\omega_P(u)$ of the associated Poincar\'{e} map $P$ defined as the time one map $P:u\mapsto \varphi(1,\cdot;u,f)$. For such Poincar\'{e} map $P$, any $\omega$-limit set $\omega_P(u)$ can be embedded into $\mathbb{R}^2$ (Tere\v{s}\v{c}\'{a}k \cite{Te} or Pol\'{a}\v{c}ik \cite{Pola}).

Sandstede and Fiedler \cite{SF1} studied the time-periodic equation \eqref{equation-1} and showed that the Poincar\'{e} map $P$ induces on any $\omega_P(u)$ a linear shift-map given by some $x$-shift $\sigma_r$, where $\sigma_r$ denote the $S^1$-action on $u\in X$ induced by shifting $x$ as $(\sigma_ru)(\cdot):=u(\cdot+r).$ Depending on whether $2\pi/r$ is rational or irrational, this induced map is periodic or ergodic. In the terminology of skew-product semiflow \eqref{equation-lim2}, the remarkable result of Sandstede and Fiedler \cite{SF1} can be reformulated as: any $\omega$-limit set $\omega(u,g)$ can be viewed as a subset of the two-dimensional torus $\mathcal{T}^1\times S^1$ carrying a linear flow (see Sandstede \cite{Sandstede}); in other words,
 $\omega(u,g)$ is imbedded into a $\mathcal{T}^1$-periodically forced circle flow on $S^1$.

 The present paper is devoted to the investigation of the
  topological structure of the $\omega$-limit set $\omega(u,g)$ of \eqref{equation-1} in time-recurrent cases including almost periodicity and almost automorphy. Based on the phenomena in autonomous and time-periodic cases (\cite{Massatt1986,Matano,SF1}), a natural general problem is:
  \begin{itemize}
 \item[{\bf (P)}]  {\it For the time-recurrent system \eqref{equation-1}, whether any $\omega(u,g)$ can be imbedded into an $H(f)$-time-recurrently forced circle flow on $S^1$?  In particular, when $f$ is uniformly almost periodic in $t$, whether $\omega(u,g)$ can be imbedded into an almost periodically forced circle flow on $S^1$?}
 \end{itemize}
Unfortunately, our example in the Appendix of this paper immediately indicates that it is not correct even for time almost periodic cases. This reveals that on this problem there are certain essential differences between time-periodic cases and non-periodic cases.

 As a consequence, it then comes out an interesting question that under what condition $\omega(u,g)$ can be imbedded into an $H(f)$-time-recurrently forced circle flow on $S^1$. In this paper, we will first try to answer this question via connecting this question to the dimension of the center space $V^c(\omega(u,g))$  associated with $\omega(u,g)$. More precisely, let $(u,g)\in X\times H(f)$ be such that the motion $\Pi^{t}(u,g)$($t\ge 0$) is bounded. Let also $\Omega=\omega(u,g)$.  Then, among others, the following results are obtained in this paper:

\begin{itemize}
 \item[(i)] (see  Theorem \ref{hyperbolic0}) {\it Assume that $\dim V^c(\Omega)=0$ (i.e., $\Omega$ is hyperbolic), then
$\Omega$ is a spatially-homogeneous  $1$-cover of $H(f)$.}

 \item[(ii)] (see Theorem \ref{norma-hyper}) A{\it ssume that $\dim V^c(\Omega)=1$. Then $\Omega$ is either spatially-homogeneous or spatially-inhomogeneous (see Definition \eqref{D:homo-inhomo}). Moreover, any spatially-inhomogeneous $\O$ can be imbedded into an $H(f)$-time-recurrently forced circle flow on $S^1$ (resp. imbedded into an almost periodically forced circle flow on $S^1$  provided that $f$ is uniformly almost-periodic in $t$).}
\end{itemize}

Conclusions (i)-(ii) indicate that, when $\dim V^c(\Omega)\le 1$, $\Omega$ is either spatially-homogeneous or spatially-inhomogeneous; and moreover, {\bf (P)} is indeed correct for any spatially-inhomogeneous $\O$ automatically when $\dim V^c(\Omega)\le 1$. On the other hand,  a careful examination yields that the counter example in the Appendix admits $\dim V^c(\Omega)=2$ (see Remark A.1(i)), which means that one can not always expect  {\bf (P)} to hold anymore when $\dim V^c(\O)>1$.

We can further characterize the structure of $\O$ under the condition that $\dim V^c(\Omega)=2$ and the dimension of the unstable space $V^u(\Omega)$ associated with $\O$ is odd. More precisely, for $u\in M\subset X$, let $\Sigma u=\{\sigma_a u\,|\, a\in S^1\}$ (resp. $\Sigma M=\cup_{u\in M}\Sigma u$) be the $S^1$-group orbit of $u$ (resp. of $M$). Then we prove

\begin{itemize}
\item[(iii)]  (see Theorem \ref{structure-thm}) {\it  Assume that $\dim V^c(\Omega)=2$ and $\dim V^u(\Omega)$ is odd. Then

(a) Either $\Sigma M_1=\Sigma M_2$ or
$\Sigma M_1\cap\Sigma M_2=\emptyset$, for any two minimal subsets $M_1,M_2\subset  \Omega$;

(b) $\O$ contains at most two minimal sets $M_1$ and $M_2$ with $\Sigma M_1\cap \Sigma M_2= \emptyset$;

(c) Given any minimal set $M\subset \Omega$, $\Omega\cap \Sigma M$ can be {\rm residually imbedded} into an $H(f)$-time-recurrently forced circle flow on $S^1$ (resp. imbedded into an {\rm almost automorphically} forced circle flow on $S^1$  if $f$ is almost periodic in $t$).}
\end{itemize}
Conclusion (iii) reveals that, for higher dimensional center space $\dim V^c(\Omega)$, the structure of the $\omega$-limit set $\O$ can be more complicated; and moreover, residually imbedding and almost automorphically forced circle flow may occur.

The above main results (i)-(iii) are generalizations from autonomous and time-periodic cases (\cite{Massatt1986,Matano,SF1}) to general systems with time-recurrent structure which includes almost periodicity and almost automorphy. It
also deserves to point out that an almost periodically (automorphically) forced circle
flow has interesting and fruitful dynamical behavior (see, e.g. \cite{HuYi,Yi} and the references therein). The new phenomena (i)-(iii) we
discovered here reinforce the appearance of the almost periodically (automorphically) forced circle
flow on the $\omega$-limit set $\O$ of
the infinite-dimensional dynamical systems generated by evolutionary equations.

Here, we also mention that, for time almost-periodic system \eqref{equation-1},  the topological structure of the minimal sets (i.e., the simplest $\omega$-limit sets) has been investigated by the present authors in \cite{SWZ} very recently. Moreover, for the reflection-symmetric nonlinearity $f(t,u,u_x)=f(t,u,-u_x)$ in \eqref{equation-1}, one may refer to the work by Chen and Matano \cite{Chen1989160} for time-periodic cases and the work by Shen et.al \cite{SWZ2} for time almost-periodic cases.

The present paper is organized as follows. In section 2, we summarize preliminary materials to be used in our proofs which include some conceptions of dynamic systems, almost-periodic (almost-automorphic) functions, properties of zero number function of the linearized system associated with \eqref{equation-1}, as well as the invariant manifolds theory for skew-product semiflows. In section 3, we list some properties of invariant sets of \eqref{equation-lim2}.  In section 4, we introduce the skew-product seimiflows  $\tilde \Pi^t$ on the quotient space induced by the spatial-shift  and present  some basic properties of $\tilde \Pi^t$. In section 5, we  present the main results of this paper, Theorems \ref{hyperbolic0}-\ref{structure-thm}. We first study the general structure of the $\omega$-limit set $\O$ for \eqref{equation-1} with $\dim V^c(\Omega)\leq 1$ or $\dim V^c(\Omega)=2$ and $\dim V^u(\Omega)$ being odd  and prove Theorem \ref{structure-thm}, and then  further study the  $\omega$-limit set $\Omega$ with $\dim V^c(\Omega)\leq 1$ and prove Theorem \ref{norma-hyper} and Theorem \ref{hyperbolic0}, respectively.

\section{Preliminaries}

In this section, we introduce some conceptions, notations and properties which will be often used in the later sections (cf. \cite{SWZ,SWZ2}).

\subsection{Some conceptions of dynamic systems}
Let $Y$ be a compact metric space with metric $d_{Y}$, and
$\sigma:Y\times \RR\to Y, (y,t)\mapsto y\cdot t$ be a continuous
flow on $Y$, denoted by $(Y,\sigma)$ or $(Y,\RR)$. A pair
$y_1,y_2$ of different elements of $Y$ are said to be {\it positively proximal} (resp. {\it negatively proximal}), if there is $t_n\to\infty$ (resp. $t_n\to-\infty$) as $n\to\infty$ such
that $d_Y(y_1\cdot t_n,y_2\cdot t_n)\to 0$, the pair $y_1,y_2$ is called {\it two sided proximal} if it is both a positively and negatively proximal pair.

Let $(Y,\mathbb{R})$, $(Z,\mathbb{R})$ be two continuous compact flows. $Z$ is called a {\it $1$-cover} ({\it almost $1$-cover}) of $Y$ if there is an onto flow homomorphism $p:Z\to Y$ such that $p^{-1}(y)$ is a singleton for any $y\in Y$ (for at least one $y\in Y$). Moreover, if $Z$ is an almost $1$-cover of $Y$, it is also called an {\it almost automorphic extension} of $Y$. Here $(Y,\mathbb{R})$ is called an {\it factor} of $(Z,\mathbb{R})$.

\subsection{Almost periodic (automorphic) functions and almost periodically (automorphically) forced circle flows}

Let $D$ be a subset of $\RR^m$. We list the following definitions and notations in this subsection.
\begin{definition}\label{admissible}
A function $f\in C(\RR\times D,\RR)$ is said to be {\it
admissible} if for any compact subset $K\subset D$, $f$ is bounded and uniformly continuous on
$\RR\times K$. $f$ is $C^r$ ($r\ge 1$) {\it admissible} if $f$ is $C^r$ in $w\in D$ and Lipschitz in $t$, and $f$ as well as its partial derivatives to order $r$ are admissible.
\end{definition}

Let $f\in C(\RR\times D,\RR)$ be an admissible function. Then
$H(f)={\rm cl}\{f\cdot\tau:\tau\in \RR\}$ (called the {\it hull of
$f$}) is compact and metrizable under the compact open topology (see \cite{Sell,Shen1998}), where $f\cdot\tau(t,\cdot)=f(t+\tau,\cdot)$. Moreover, the time translation $g\cdot t$ of $g\in H(f)$ induces a natural
flow on $H(f)$ (cf. \cite{Sell}).

\begin{definition}\label{almost}
{\rm
\begin{itemize}
\item[(1)] A function $f\in C(\RR,\RR)$ is {\it recurrent} if $H(f)$ is minimal under the time translation flow $(t,g)\mapsto g\cdot t$ for $t\in\RR$ and $g\in H(f)$.

\item[(2)] A function $f\in C(\RR,\RR)$ is {\it almost automorphic} if for every $\{t'_k\}\subset\mathbb{R}$ there is a subsequence $\{t_k\}$
and a function $g:\mathbb{R}\to \mathbb{R}$ such that $f(t+t_k)\to g(t)$ and $g(t-t_k)\to f(t)$ pointwise.

\item[(3)]   $f$ is {\it almost periodic} if for any sequence $\{t_n\}$ there is a subsequence $\{t_{n_k}\}$ such that $\{f(t+t_{n_k})\}$ converges uniformly.

\item[(4)] A function $f\in C(\RR\times D,\RR)(D\subset \RR^m)$ is {\it uniformly recurrent in $t$} (resp. {\it uniformly almost automorphic in $t$}, {\it  uniformly almost periodic in $t$}) , if $f$ is both admissible and, for each fixed $d\in D$, $f(t,d)$ is  recurrent (resp. almost automorphic, almost periodic)  with respect to $t\in \RR$.
\end{itemize}
}
\end{definition}

\begin{remark}\label{a-p-to-minial}
{\rm  If $f$ is a uniformly almost periodic (automorphic) function in $t$, then $H(f)$ is always {\it minimal}, we call $(H(f),\mathbb{R})$ an almost periodic (automorphic) minimal flow.
 Moreover,  $g$ is a uniformly almost periodic (automorphic) function for all  (residually many)
$g\in H(f)$ (see, e.g. \cite{Shen1998}).}
\end{remark}
\begin{definition}
{\rm
Let $(Y,\sigma)$ be a flow on the compact metric space $Y$.  A skew-product circle flow  $\Lambda^t:S^1\times Y\rightarrow S^1\times Y$
is a skew-product flow of the following form
 \begin{equation}\label{skew-product-circleflow}
 \Lambda^{t}(u,y)=(\varphi(t,u,y),y\cdot t),\quad t\in\mathbb{R},\, (u,y)\in S^1\times Y.
 \end{equation}
 If $(Y,\sigma)$ is a (an almost periodic or almost automorphic) minimal flow, then $\Lambda^t$ is called {\it a time recurrently} ({\it an almost periodically or almost automorphically}) forced circle flow.
}
\end{definition}

\subsection{Zero number function}
We now recall the zero number function on $S^1$ and list some related properties.

Given a $C^{1}$-smooth function $u:S^{1}\rightarrow \mathbb{R}$, the zero number of $u$ is
 defined as
$$z(u(\cdot))={\rm card}\{x\in S^{1}|u(x)=0\}.$$
The following key lemma describes the behavior of the zero number for linear non-autonomous parabolic equations and was originally presented in \cite{2038390,H.MATANO:1982} and improved in \cite{Chen98}.
\begin{lemma}\label{zero-number}
Let $\varphi(t,\cdot)$ be a classical nontrivial solution of
\begin{equation}
\begin{cases}
\varphi_{t}=a(t,x)\varphi_{xx}+b(t,x) \varphi_{x}+c(t,x)\varphi,\quad x\in S^1,\\
\varphi_{0}=\varphi(0,\cdot)\in H^{1}(S^{1}),
\end{cases}
\end{equation}
where $a,a_t,a_x,b$ and $c$ are bounded continuous functions, $a\ge \delta >0$. Then the following properties hold.
\par
{\rm (a)} $z(\varphi(t,\cdot))<\infty$ for $t>0$ and is non-increasing in $t$.\par
{\rm (b)}  $z(\varphi(t,\cdot))$ can drop only at $t_{0}$ such that $\varphi(t_{0},\cdot)$ has a
multiple zero on $S^{1}$.\par
{\rm (c)}  $z(\varphi(t,\cdot))$ can drop only finite many times, and there exists a $T>0$
such that $\varphi(t,\cdot)$ has only simple zeros on $S^{1}$ as $t\geq T$(hence
$z(\varphi(t,\cdot))=\mathrm{constant}$ as $t\geq T$).
\end{lemma}

\begin{corollary}\label{difference-lapnumber}
For any $g\in H(f)$, let $\varphi(t,\cdot;u,g)$ and $\varphi(t,\cdot;\hat{u},g)$ be two
distinct solutions of {\rm (\ref{equation-lim1})} on
$ \mathbb{R}^+$. Then

{\rm (a)} $z(\varphi(t,\cdot;u,g)-\varphi(t,\cdot;\hat{u},g))<\infty$ for $t>0$ and is non-increasing in t;

{\rm (b)} $z(\varphi(t,\cdot;u,g)-\varphi(t,\cdot;\hat{u},g))$
strictly decreases at $t_0$ such that the function $\varphi(t_0,\cdot;u,g)-\varphi(t_0,\cdot;\hat{u},g)$ has a multiple
zero on $S^1$;

{\rm (c)} $z(\varphi(t,\cdot;u,g)-\varphi(t,\cdot;\hat{u},g))$ can drop only finite many times, and there exists a $T>0$ such that  $$z(\varphi(t,\cdot;u,g)-\varphi(t,\cdot;\hat{u},g))\equiv
\textnormal{constant}$$ for all $t\ge T$.

\end{corollary}

\begin{lemma}\label{zero-cons-local}
Let $u\in X$ be such that $u$ has only simple zeros on $S^1$, then there exists a $\delta>0$ such that for any $v\in X$ with $\|v\|<\delta$, one has
  \begin{equation*}
    z(u)=z(u+v).
  \end{equation*}
\end{lemma}
\begin{proof}
See Corollary 2.1 in \cite{SF1} or Lemma 2.3 in \cite{Chen1989160}.
\end{proof}

The proof of the following lemma can be found in \cite[Lemma 2.4]{SWZ}.
\begin{lemma}\label{sequence-limit}
Fix $g,\ g_{0}\in H(f)$. Let $(u^{i},g)\in p^{-1}(g),(u_{0}^{i},g_{0})\in p^{-1}(g_{0})$  $ (i=1,\ 2,\ u^{1}\neq u^{2},\ u_{0}^{1}\neq u_{0}^{2})$ be such that $\Pi^{t}(u^{i},g)$ is defined on $\mathbb{R}^{+}$ (resp. $\mathbb{R}^-$) and $\Pi^{t}(u_{0}^{i},g_{0})$ is defined on $\mathbb{R}$. If there exists a sequence $t_{n}\rightarrow +\infty$ (resp. $s_{n}\rightarrow -\infty$) as $n\rightarrow \infty$, such that $\Pi^{t_{n}}(u^{i},g)\rightarrow (u_{0}^{i},g_{0})$ (resp. $\Pi^{s_{n}}(u^{i},g)\rightarrow (u_{0}^{i},g_{0})$) as $n\rightarrow \infty (i=1,2)$, then
$$z(\varphi(t,\cdot;u_{0}^{1},g_{0})-\varphi(t,\cdot;u_{0}^{2},g_{0}))\equiv \textnormal{constant},$$
for all $t\in \mathbb{R}$.
\end{lemma}

\subsection{Invariant subspaces and invariant manifolds of parabolic equations on the circle}

Let $E\subset X\times H(f)$ be a connected and compact invariant set of \eqref{equation-lim2} which admits a compact flow extension. Denote by $\sigma(E)$ the Sacker-Sell spectrum associated with $E$. Then $\sigma(E)=\cup_{k=0}^\infty I_k$, where $I_k=[a_k,b_k]$ and $\{I_k\}$ is ordered from right to left, that is, $\cdots<a_k\leq b_k<a_{k-1}\leq b_{k-1}<\cdots<a_0\leq b_0$ (cf. \cite{Chow1994,Sacker1978,Sacker1991}).

 Consider the linearly variational equation of \eqref{equation-lim1}:
\begin{equation}\label{linear-equation2}
\psi_t=\psi_{xx}+a(x,\omega\cdot t)\psi_x+b(x,\omega\cdot t)\psi,\,\,t>0,\,x\in S^{1}=\mathbb{R}/2\pi \mathbb{Z},
\end{equation}
where  $\omega=(u_0,g)\in E$, $a(x,\omega)=g_p(0,u_0,(u_0)_x)$ (here $g_p(\cdot,\cdot,p)$ is the derivative of $g$ with respect to $p$), $b(x,\omega)=g_u(0,u_0,(u_0)_x)$.

Let $\Psi(t,\omega):X\rightarrow X$ be the evolution operator generated by \eqref{linear-equation2}, that is, the evolution operator of the following equation:
\begin{equation}\label{linear-opera}
  v'=A(\omega\cdot t)v,\quad t>0,\,\omega\in E,\, v\in X,
\end{equation}
where $A(\omega)v=v_{xx}+a(x,\omega)v_x+b(x,\omega)v$, and $\omega\cdot t$ is as in \eqref{linear-equation2}.

For any given $0\leq n_1\leq n_2\leq\infty$. When $n_2\neq \infty$, let
\begin{equation*}\label{twoside-estimate}
\begin{split}
V^{n_1,n_2}(\omega)=\{v\in X:&\|\Psi(t,\omega)v\|=o(e^{a^-t})\ \text{as}\ t\rightarrow -\infty\\ & \|\Psi(t,\omega)v\|=o(e^{b^+t})\ \text{as}\ t\rightarrow \infty\}
\end{split}
\end{equation*}
where $a^-$, $b^+$ are such that $b_{n_2+1}<a^-<a_{n_2}\leq b_{n_1}<b^+<a_{n_1-1}$. Here $a_{n_1-1}=\infty$ if $n_1=0$.  When $n_1<n_2=\infty$, let
\[
V^{n_1,\infty}(\omega)=\{v\in X :\|\Psi(t,\omega)v\|=o(e^{b^+t})\text{as }t\to\infty\}
\]
where $b^+$ is such that $b_{n_1}<b^+<\lambda$ for any $\lambda\in \cup_{k=0}^{n_1-1}I_k$.

The following lemma is adopted from \cite[Lemma 2.6]{SWZ}, which directly follows from the Floquet theory established by Chow, Lu and Mallet-Paret in \cite[Sections 4 and 9]{Chow1995} (see also in \cite{Te} or \cite[Theorem 4.5]{Pola}).
\begin{lemma}\label{L:zero-inva}
For given $0\leq n_1\leq n_2\leq\infty$($n_1\not = n_2$ when $n_2=\infty$), we have $N_1\leq z(v(\cdot))\leq N_2$ for any $v\in V^{n_1,n_2}(\omega)$, where
\begin{equation*}
N_1=\left\{
\begin{split}
 &{\rm dim}V^{0,n_1-1},\,\quad\,\,\,\text{ if }{\rm dim}V^{0,n_1-1}\text{ is even;}\\
 &{\rm dim}V^{0,n_1-1}+1,\,\text{ if }{\rm dim}V^{0,n_1-1}\text{ is odd,}
\end{split}\right.
\end{equation*} and
\begin{equation*}
N_2=\left\{
\begin{split}
 &{\rm dim}V^{0,n_2},\,\quad\,\,\,\text{ if }{\rm dim}V^{0,n_2}\text{ is even;}\\
 &{\rm dim}V^{0,n_2}-1,\,\text{ if }{\rm dim}V^{0,n_2}\text{ is odd.}
\end{split}\right.
\end{equation*}
\end{lemma}

\vskip 2mm
By using arguments as in \cite{P.Bates,Chow1991, Chow1994-2, Hen, SWZ,SWZ2}, we have the following lemma concerning with nonlinear invariant manifolds.
\begin{lemma}\label{invari-mani}
 There is a $\delta_0>0$ such that for any $0<\delta^*<\delta_0$ and $0\leq n_1\leq n_2\leq\infty$ ($n_1\not = n_2$ when $n_2=\infty$), \eqref{equation-lim1} admits for each $\omega=(u_0,g)\in E$ a local invariant manifold  $M^{n_1,n_2}(\omega,\delta^*)$ with the following properties:
 \begin{itemize}
 \item [\rm{(i)}] There are $K_0>0$, and a bounded continuous function $h^{n_1,n_2}(\omega): V^{n_1,n_2}(\omega)$ $\rightarrow V^{n_2+1,\infty}(\omega)\oplus V^{0,n_1-1}(\omega)) $ being $C^1$ for each fixed $\omega\in E$, and $h^{n_1,n_2}(v,\omega)$ $=o(\|v\|)$, $\|(\partial h^{n_1,n_2}/\partial v)(v,\omega)\|\leq K_0$ for all $\omega\in E$, $v\in V^{n_1,n_2}(\omega)$ such that
 \begin{eqnarray*}
 M^{n_1,n_2}(\omega,\delta^*)=\left\{ u_0+v_0^{n_1,n_2}+h^{n_1,n_2}(v_0^{n_1,n_2},\omega):v_0^{n_1,n_2}\in V^{n_1,n_2}(\omega) \cap \{v\in X: \|v\|<\delta^*\}\right\}.
 \end{eqnarray*}
 Moreover, $M^{n_1,n_2}(\omega,\delta^*)-u_0$ are diffeomorphic to $V^{n_1,n_2}(\omega)\cap\{v\in X| \|v\|<\delta^*\}$, and tangent to $V^{n_1,n_2}(\omega)$ at $0\in X$ for each $\omega\in E$.
 \item [\rm{(ii)}] $M^{n_1,n_2}(\omega,\delta^*)$ is locally invariant in the sense that if $v\in M^{n_1,n_2}(\omega,\delta^*)$ and $\norm{\varphi(t,\cdot;v,g)-\varphi(t,\cdot;u_0,g)}<\delta^*$ for all $t\in [0,T]$, then $\varphi(t,\cdot;v,g)\in M^{n_1,n_2}(\omega\cdot t,\delta^*)$ for all $t\in [0,T]$. Therefore,
 for any $v\in M^{n_1,n_2}(\omega,\delta^*)$, there is a $\tau>0$ such that $\varphi(t,\cdot;v,g)\in M^{n_1,n_2}(\omega\cdot t,\delta^*)$ for any $t\in \mathbb{R}$ with $0<t<\tau$.
 \end{itemize}
\end{lemma}
 \par

Suppose that $0\in \sigma(E)$ and $n_0$ is such that $0\in I_{n_0}\subset\sigma(E)$. Then $V^s(\omega)=V^{n_0+1,\infty}(\omega)$, $V^{cs}(\omega)=V^{n_0,\infty}(\omega)$, $V^{c}(\omega)=V^{n_0,n_0}(\omega)$, $V^{cu}(\omega)=V^{0,n_0}(\omega)$, and $V^u(\omega)=V^{0,n_0-1}(\omega)$ are referred to as {\it stable, center stable, center, center unstable}, and {\it unstable subspaces} of \eqref{linear-equation2} at $\omega\in E$, respectively. And $M^{cs}(\omega,\delta^*)=M^{n_0,\infty}(\omega,\delta^*)$, $M^c(\omega,\delta^*)=M^{n_0,n_0}(\omega,\delta^*)$, $M^{cu}(\omega,\delta^*)=M^{0,n_0}(\omega,\delta^*)$, and $M^u(\omega,\delta^*)=M^{0,n_0-1}(\omega,\delta^*)$ are referred to as {\it local stable, center stable, center, center unstable, and unstable manifolds} of \eqref{equation-lim1} at $\omega\in E$, respectively.

We now list some useful properties of local invariant manifolds which can be found in \cite{SWZ,SWZ2}.

\begin{remark}\label{stable-leaf}
{\rm
(1) $M^s(\omega,\delta^*)$ and $M^u(\omega,\delta^*)$ are overflowing invariant in the sense that if $\delta^*$ is sufficiently small, then
\[
\varphi(t,\cdot;M^s(\omega,\delta^*),g)\subset M^s(\omega\cdot t,\delta^*),
\]
for $t$ sufficiently positive, and
\[
\varphi(t,\cdot;M^u(\omega,\delta^*),g)\subset M^u(\omega\cdot t,\delta^*),
\]
for $t$ sufficiently negative.  $M^s(\omega,\delta^*)$ and $M^u(\omega,\delta^*)$  are unique and have the following characterizations:
there are $\delta_1^*,\delta_2^*>0$ such that
\begin{align*}
&\{v\in X\, :\, \|\varphi(t,\cdot;v,g)-\varphi(t,\cdot;u,g)\|\le \delta_1^*\,\,{\rm for}\,\, t\ge 0\,\textnormal{ and }\varphi(t,\cdot;v,g)-\varphi(t,\cdot;u,g)\to 0\,\\ &\textnormal{exponentially as}\,\, t\to\infty\}\\
&\subset
M^s(\omega,\delta^*)\subset \{v\in X\, :\, \|v-u\|\le \delta_2^*,\,\, \|\varphi(t,\cdot;v,g)-\varphi(t,\cdot;u,g)\|\to 0\,\, {\rm as}\,\, t\to\infty\}
\end{align*}
and
\begin{align*}
&\{v\in X\, :\,\textnormal{the backward orbit } \varphi(t,\cdot;v,g) \textnormal{ exists and } \|\varphi(t,\cdot;v,g)-\varphi(t,\cdot;u,g)\|\le \delta_1^*\,\,{\rm for}\,\, t\le 0,\\
&\textnormal{ further, }\varphi(t,\cdot;v,g)-\varphi(t,\cdot;u,g)\to 0\,\, \textnormal{exponentially as}\,\, t\to -\infty\}\\
&\subset
M^u(\omega,\delta^*)\subset \{v\in X\, :\, \|v-u\|\le \delta_2^*,\, \|\varphi(t,\cdot;v,g)-\varphi(t,\cdot;u,g)\|\to 0\,\, {\rm as}\,\, t\to -\infty\}.
\end{align*}

Moreover, one can find constants $\alpha$, $C>0$, such that for any $\omega\in E$, $v^s\in M^s(\omega,\delta^*)$, $v^u\in M^u(\omega,\delta^*)$,
\begin{equation}\label{exponen-decrea}
\begin{split}
  \|\varphi(t,\cdot;v^s,g)-\varphi(t,\cdot;u,g)\|&\leq Ce^{-\frac{\alpha}{2}t}\|v^s-u\|\quad \text{for}\ t\geq 0,\\
  \|\varphi(t,\cdot;v^u,g)-\varphi(t,\cdot;u,g)\|&\leq Ce^{\frac{\alpha}{2}t}\|v^u-u\|\quad \text{for}\ t\leq 0.
\end{split}
\end{equation}

\vskip 3mm
(2) $M^{cs}(\omega,\delta^*)$ (choose $\delta^*$ smaller if necessary) has a repulsion property in the sense that if $\norm{v-u}<\delta^*$ but $v\notin M^{cs}(\omega,\delta^*)$, then there is $T>0$ such that $\norm{\varphi(T,\cdot;v,g)-\varphi(T,\cdot;u,g)}\ge \delta^*$. Consequently, if $\norm{\varphi(t,\cdot;v,g)-\varphi(t,\cdot;u,g)}<\delta^*$ for all $t\ge 0$ then one may conclude that $v\in M^{cs}(\omega,\delta^*)$. Note that $M^{cs}(\omega,\delta^*)$ is not unique in general.

\vskip 3mm

(3) $M^{cu}(\omega,\delta^*)$ has an attracting property  in the sense that if $\|\varphi(t,\cdot;v,g)-\varphi(t,\cdot;u,g)\|<\delta^*$ for all $t\ge 0$, then $v^*\in M^{cu}(\omega^*,\delta^*)$ whenever $(\varphi(t_n,\cdot;v,g),\omega\cdot t_n)\to (v^*,\omega^*)$ and with some $t_n\to \infty$. Moreover, one can choose $\delta^*$ smaller such that, if  $\norm{v-u}<\delta^*$ with a unique backward orbit $\varphi(t,\cdot;v,g)(t\le 0)$ but $v\notin M^{cu}(\omega,\delta^*)$, then there is $T<0$ such that $\norm{\varphi(t,\cdot;v,g)-\varphi(t,\cdot;u,g)}\ge \delta^*$. As a consequence, if $v$ has a unique backward orbit $\varphi(t,\cdot;v,g)(t\le 0)$ with $\norm{\varphi(t,\cdot;v,g)-\varphi(t,\cdot;u,g)}<\delta^*$ for all $t\le 0$, then one may conclude that $v\in M^{cu}(\omega,\delta^*)$. Note that $M^{cu}(\omega,\delta^*)$ is not unique in general.

\vskip 3mm
(4) For any $\omega\in E$, we have
\[
 M^{cs}(\omega,\delta^*)={\cup}_{u_c\in M^c(\omega,\delta^*)}\bar{M}_s(u_c,\omega,\delta^*)\ ({\rm resp. }\, \ M^{cu}(\omega,\delta^*)={\cup}_{u_c\in M^c(\omega,\delta^*)}\bar{M}_u(u_c,\omega,\delta^*)),
\]
where $\bar{M}_s(u_c,\omega,\delta^*)$ (resp. $\bar{M}_u(u_c,\omega,\delta^*)$) is the so-called {\it stable leaf} (resp. {\it unstable leaf}) of \eqref{equation-lim1} at $u_c$. It is invariant in the sense that if $\tau>0$ (resp. $\tau<0$) is such that $\varphi(t,\cdot;u_c,g)\in M^{c}(\omega\cdot t,\delta^*)$ and $\varphi(t,\cdot;v,g)\in M^{cs}(\omega,\delta^*)$ (resp. $\varphi(t,\cdot;v,g)\in M^{cu}(\omega,\delta^*)$) for all $0\leq t<\tau$ (resp. $\tau<t\leq 0$), where $v\in \bar{M}_s(u_c,\omega,\delta^*)$ (resp. $v\in \bar{M}_u(u_c,\omega,\delta^*)$), then $\varphi(t,\cdot;v,g)\in\bar{M}_s(\varphi(t,\cdot;u_c,g),\omega\cdot t,\delta^*)$ (resp. $\varphi(t,\cdot;v,g)\in\bar{M}_u(\varphi(t,\cdot;u_c,g),\omega\cdot t,\delta^*)$) for $0\leq t<\tau$ (resp. $\tau<t\leq 0$). Moreover, there are $K,\beta >0$ such that for any $u\in\bar{M}_s(u_c,\omega,\delta^*)$ (resp. $u\in\bar{M}_u(u_c,\omega,\delta^*)$) and $\tau>0$ (resp. $\tau<0$) with $\varphi(t,\cdot;v,g)\in M^{cs}(\omega\cdot t,\delta^*)$ (resp. $\varphi(t,\cdot;v,g)\in M^{cu}(\omega\cdot t,\delta^*)$), $\varphi(t,\cdot;u_c,g)\in M^{c}(\omega\cdot t,\delta^*)$ for $0\leq t<\tau$ (resp. $\tau<t\leq 0$), one has that
\begin{equation*}
\begin{split}
\|\varphi(t,\cdot;v,g)-\varphi(t,\cdot;u_c,g)\|&\le  Ke^{-\beta t}\|v-u_c\|\\
(\mathrm{resp}. \ \|\varphi(t,\cdot;v,g)-\varphi(t,\cdot;u_c,g)\|&\le  Ke^{\beta t}\|v-u_c\|)
\end{split}
\end{equation*}
for $0\leq t<\tau$ (resp. $\tau< t\leq 0$).

}
\end{remark}

\begin{lemma}\label{zerocenter}
  Let $\omega=(u_0,g)\in E$ and \begin{equation*}
N_u=\left\{
\begin{split}
 &{\dim} V^u(E),\,\quad\,\,\,\text{ if }{\rm dim}V^u(E)\text{ is even,}\\
 &{\rm dim}V^u(E)+1,\,\text{ if }{\rm dim}V^u(E)\text{ is odd.}
\end{split}\right.
\end{equation*}
{ Suppose that  ${\rm dim}V^u(E)\ge 1$,
then} for $\delta^*>0$ small enough, one has
\begin{itemize}

\item[{\rm (1)}] If ${\rm dim}V^c(E)=0$ and ${\rm dim}V^u(E)$ is odd, then
\begin{eqnarray*}
\begin{split}
 & z(u(\cdot)-u_0(\cdot))\geq N_u\quad \text{for }u\in M^{s}(\omega,\delta^*)\setminus\{u_0\},\\
 & z(u(\cdot)-u_0(\cdot))\leq N_u-2\quad \text{for }u\in M^{u}(\omega,\delta^*)\setminus\{u_0\}.\\
\end{split}
\end{eqnarray*}

\item[{\rm (2)}] If ${\rm dim}V^c(E)=1$ and $\dim V^u(E)$ is odd, then
\begin{eqnarray*}
\begin{split}
 & z(u(\cdot)-u_0(\cdot))\geq N_u\quad \text{for }u\in M^{cs}(\omega,\delta^*)\setminus\{u_0\},\\
 & z(u(\cdot)-u_0(\cdot))=N_u\quad \text{for }u\in M^{c}(\omega,\delta^*)\setminus\{u_0\},\\
 & z(u(\cdot)-u_0(\cdot))\leq N_u-2\quad \text{for }u\in M^{u}(\omega,\delta^*)\setminus\{u_0\}.\\
\end{split}
\end{eqnarray*}

\item[{\rm (3)}] If ${\rm dim}V^c(E)=1$ and $\dim V^u(E)$ is even, then
\begin{eqnarray*}
\begin{split}
 & z(u(\cdot)-u_0(\cdot))\geq N_u+2\quad \text{for }u\in M^{s}(\omega,\delta^*)\setminus\{u_0\},\\
 & z(u(\cdot)-u_0(\cdot))=N_u\quad \text{for }u\in M^{c}(\omega,\delta^*)\setminus\{u_0\},\\
 & z(u(\cdot)-u_0(\cdot))\leq N_u\quad \text{for }u\in M^{cu}(\omega,\delta^*)\setminus\{u_0\}.\\
\end{split}
\end{eqnarray*}
\end{itemize}
\end{lemma}
\begin{proof}
  See \cite[Corollary 3.5]{SWZ}.
\end{proof}

\begin{lemma}\label{hyperbolic1}
\begin{description}
\item[{\rm (i)}] Suppose that $0\notin \sigma (E)$. Then, for $(u_1,g)$, $(u_2,g)\in E$ with $\|u_1-u_2\|\ll1$, one has $M^s(u_1,g,\delta^*)\cap M^u(u_2,g,\delta^*)\neq \emptyset$ and $M^u(u_1,g,\delta^*)\cap M^s(u_2,g,\delta^*)\neq \emptyset$.

\item[{\rm (ii)}] Suppose that $0\in \sigma (E)$. Then, for $(u_1,g)$, $(u_2,g)\in E$ with $\|u_1-u_2\|\ll1$, one has $M^{cs}(u_1,g,\delta^*)\cap M^u(u_2,g,\delta^*)\neq \emptyset$ and $M^{s}(u_1,g,\delta^*)\cap M^{cu}(u_2,g,\delta^*)\neq \emptyset$.
\end{description}
\end{lemma}
\begin{proof}
  See \cite[Lemma 3.7]{SWZ}.
\end{proof}

\begin{remark}\label{invari-space}
{\rm
For any minimal set $M\subset E$, one has $\sigma(M)\subset \sigma(E)$ and $\dim V^u(M)\geq\dim V^u(E)$, $\dim V^c(M)\leq\dim V^c(E)$ and $\mathrm{codim}V^s(M)\leq \mathrm{codim}V^s(E)$ (here $V^u(M)$, $V^c(M)$ and $V^s(M)$ are {\it stable space, center space and unstable space} of the linearized variational equation of \eqref{equation-lim1} on $M$).}
\end{remark}

\section{Basic structural properties of invariant sets}

In this section, we present some basic properties of invariant sets, in particular, $\omega$-limit sets and minimal sets, of the skew-product semiflow \eqref{equation-lim2}. Throughout this section, $E$ denotes  a connected and compact invariant set of \eqref{equation-lim2}, $M$  is a minimal set of \eqref{equation-lim2},  and $\Omega:=\omega(u,g)$ denotes an $\omega$-limit set of \eqref{equation-lim2}.

Hereafter, we always assume that $X$ is the fractional power space as defined in the introduction.
Given any $u\in X$ and $a\in S^1$, we define the shift $\sigma_a$ on $u$ as $(\sigma_a u)(\cdot)=u(\cdot+a)$.
So, if $\varphi(t,\cdot;u,g)$ is a classical solution of \eqref{equation-lim1}, then it is easy to check that $\sigma_a\varphi(t,\cdot;u,g)$ is a classical solution of \eqref{equation-lim1}. Moreover, the uniqueness of solution ensures the {\it translation invariance}, that is, $\sigma_a\varphi(t,\cdot;u,g)=\varphi(t,\cdot;\sigma_au,g)$.

 Let $u\in A\subset X$, we write
 \begin{equation}\label{E:group-orbit-11}
 \Sigma u=\{\sigma_a u\,|\, a\in S^1\}
 \end{equation} as the {\it $S^1$-group orbit} of $u$, and write $\sigma_a A=\{\sigma_au|u\in A\}$ and $\Sigma A=\cup_{u\in A}\Sigma u$, respectively.

 \vskip 2mm

The following two lemmas are concerning with some useful properties of the invariant set $E$.

\begin{lemma}\label{traslation}
Let $E\subset X\times H(f)$ be a connected and compact invariant set of \eqref{equation-lim2}. Then, for any $a\in S^1$, one has $\dim V^u(\sigma_a E)=\dim V^u(E)$, $\dim V^c(\sigma_a E)=\dim V^c(E)$ and $\mathrm{codim} V^s(\sigma_a E)=\mathrm{codim} V^s(E)$.
\end{lemma}
\begin{proof}
  It follows directly from the translation invariance and the definition of Sacker-Sell spectrum on $E$.
\end{proof}

\begin{lemma}\label{hyperbolic2}
Assume that $\dim V^c(E)=0$ and $\dim V^u(E)>0$. Then $\dim V^u(E)$ is odd; and moreover, $E$ does not contain any two sided proximal pair.
\end{lemma}
\begin{proof}
 It can be proved by the similar arguments as those in  \cite[Lemmas 4.5 and 4.6]{SWZ2}.
\end{proof}

\begin{remark}
{\rm
It deserves to point out that all the statements in \cite[Section 4]{SWZ2} are still valid in our present setting without the reflection symmetry of $f$, except for \cite[Theorem 4.1]{SWZ2}.
}
\end{remark}

Before going further, we give the following definition:

\begin{definition}\label{D:homo-inhomo}
{\rm  A point $u\in X$ is called {\it spatially-homogeneous} if $u(\cdot)$ is independent of the spatial variable $x$. Otherwise, $u$ is called {\it spatially-inhomogeneous}. A subset $A\subset X$ is called {\it spatially-homogeneous} (resp. {\it spatially-inhomogeneous}) if any point in $A$ is spatially-homogeneous (resp. spatially-inhomogeneous).
}
\end{definition}
It is not difficult to see that any minimal set $M$ is either spatially-inhomogeneous; or otherwise, $M$ is spatially-homogeneous.

The following two lemmas {\it summarize} some interesting properties of the minimal set $M$.

\begin{lemma}\label{hyperbolic-minimal}
\begin{itemize}
\item[{\rm (1)}] If $\dim V^c(M)=0$. Then $M$ is spatially-homogeneous and $1$-cover of $H(f)$.

\item[{\rm (2)}] If $\dim V^c(M)=1$, then $M$ is spatially-homogeneous if and only if $\dim V^u(M)=0$. Moreover, if $\dim V^c(M)=1$ and $\dim V^u(M)=0$, then $M$ is an almost $1$-cover of $H(f)$.
\end{itemize}
\end{lemma}

\begin{proof}
See \cite[Theorem 4.1]{SWZ} or \cite[Lemma 4.2]{SWZ2} for (1); and see \cite[Lemma 5.1]{SWZ2} for (2). Here, we emphasize that  the proof of \cite[Lemma 5.1]{SWZ2} is only based on  \cite[Theorem 5.1 (ii)]{SWZ2}; while a careful examination yields that \cite[Theorem 5.1 (ii)]{SWZ2} is still valid for $f$ without reflection symmetry.
\end{proof}

 Hereafter, we write $m(u)=\max_{x\in S^1}u(x)$ as the maximal value of $u\in X$ on $S^1$.

\begin{lemma}\label{normhyper-minimal}
Assume that $\dim V^c(M)=1$, or $\dim V^c(M)=2$ with $\dim V^u(M)$ being odd. Then the following hold:
  \begin{itemize}
  \item[\rm (i)] There is a residual invariant set $Y_0\subset H(f)$, such that for any $g\in Y_0$, there exists $u_g\in X$ such that $p^{-1}(g)\cap M\subset (\Sigma u_g,g)$.
   \item[\rm (ii)] If $\dim V^c(M)=1$ with $\dim V^u(M)>0$ (hence $M$ is spatially-inhomogeneous by Lemma \ref{hyperbolic-minimal}{\rm (2)}), then one has $Y_0=H(f)$.
   \item[\rm (iii)] For any $(u,g),(v,g)\in M$ and $a\in S^1$ with $\sigma_a u\ne v$, one has
  \begin{equation*}
    z(\varphi(t,\cdot;\sigma_a u,g)-\varphi(t,\cdot;v,g))=N_u \quad \text{for  all}\ t\in\mathbb{R},
  \end{equation*}
where
  \begin{equation}\label{E:Nu-def}
N_u=\left\{
\begin{split}
 &{\dim} V^u(M),\,\quad\,\,\,\text{ if }{\rm dim}V^u(M)\text{ is even,}\\
 &{\rm dim}V^u(M)+1,\,\text{ if }{\rm dim}V^u(M)\text{ is odd.}
\end{split}\right.
\end{equation}
\item[\rm (iv)] For any $(u,g),(v,g)\in M$, $m(u)=m(v)\Leftrightarrow ([u],g)=([v],g)$.
  \end{itemize}

\end{lemma}

\begin{proof}
  See \cite[Theorem 3.1]{SWZ} for (i)-(ii) and \cite[Corollary 3.9]{SWZ} for (iii)-(iv).
\end{proof}

\medskip

Now we are focusing on the $\omega$-limit set $\O$. For convenience, we introduce the following standing assumptions:
\smallskip

\noindent {\bf (H0)} {\it  $\dim V^c(\Omega)=0$, that is, $\Omega$ is hyperbolic.}

\smallskip

\noindent {\bf (H1)} {\it  $\dim V^c(\Omega)=1$.}

\smallskip

\noindent {\bf (H2)} {\it $\dim V^c(\Omega)=2$ and $\dim V^u(\Omega)$ is odd.}

\medskip

\begin{lemma}\label{homogeneous}
  Assume {\bf (H1)} and $\dim V^u(\Omega)>0$. Let $M\subset\Omega$ be a minimal set. Then $\dim V^c(M)\le 1$ and $\dim V^u(M)>0$. Furthermore,

 {\rm (a)} If $\dim V^c(M)=1$, then  $M$ is spatially-inhomogeneous;
 and moreover, there is $\delta^*>0$ such that $M^c(\omega,\delta^*)\subset \Sigma u$ for any $\omega=(u,g)\in M$.

 {\rm (b)} If $\dim V^c(M)=0$, then $M$ is a spatially-homogeneous $1$-cover of $H(f)$; and moreover, one has
  \begin{equation*}
\left\{
\begin{split}
 &\dim V^u(M)=\dim V^u(\Omega)\textnormal{ and }\mathrm{codim} V^s(M)=\mathrm{codim}V^s(\Omega)-1,\,\quad\,\text{ if }{\rm dim}V^u(\Omega)\text{ is odd;}\\
 &\dim V^u(M)=\dim V^u(\Omega)+1\textnormal{ and } \mathrm{codim} V^s(M)=\mathrm{codim}V^s(\Omega),\quad\,\text{ if }{\rm dim}V^u(\Omega)\text{ is even.}
\end{split}\right.
\end{equation*}
\end{lemma}

\begin{proof}
 It can be proved by the similar arguments in \cite[Lemma 5.2]{SWZ2}. Here, one needs to note that in item (a), $M$ is not necessarily a $1$-cover of $H(f)$.
\end{proof}
\begin{lemma}\label{proximalcons}
Assume {\bf (H1)} and $\dim V^u(\Omega)>0$. Let $M\subset\Omega$ be a minimal set. Then for any $(u_1,g)\in M$ and $(u_2,g)\in \Omega\setminus M$,  $\{(u_1,g),(u_2,g)\}$ is not two sided proximal pair.
\end{lemma}

\begin{proof}
It can be proved by the similar arguments as those in  \cite[Lemma 5.3]{SWZ2}. Here, it also deserves to point out that, in the proof of \cite[Lemma 5.3]{SWZ2}, $M$ is actually not needed to be a $1$-cover of $H(f)$ whenever it is spatially-inhomogeneous.
\end{proof}
\begin{lemma}\label{homogeneous-2d}
  Assume that {\bf (H2)} holds and $M\subset\Omega$ is a minimal set. Then $\dim V^c(M)\le 2$ and $\dim V^u(M)>0$.  Furthermore,

 {\rm (a)} If $\dim V^c(M)=1$, then $M$ is spatially-inhomogeneous and one has
 \begin{equation}\label{dimension1}
\begin{split}
 &\dim V^u(M)=\dim V^u(\Omega)\textnormal{ and }\mathrm{codim} V^s(M)=\mathrm{codim}V^s(\Omega)-1,\quad {\rm or}\\
 &\dim V^u(M)=\dim V^u(\Omega)+1\textnormal{ and } \mathrm{codim} V^s(M)=\mathrm{codim}V^s(\Omega);
\end{split}
\end{equation}
and moreover, there is $\delta^*>0$ such that $M^c(\omega,\delta^*)\subset \Sigma u$ for any $\omega=(u,g)\in M$.

 {\rm (b)} If $\dim V^c(M)=0$,  then $M$ is a spatially-homogeneous $1$-cover of $H(f)$; and moreover, one has
 \begin{equation}\label{dimension2}
\begin{split}
 &\dim V^u(M)=\dim V^u(\Omega)\textnormal{ and }\mathrm{codim} V^s(M)=\mathrm{codim}V^s(\Omega)-2,\quad {\rm or}\\
 &\dim V^u(M)=\dim V^u(\Omega)+2\textnormal{ and } \mathrm{codim} V^s(M)=\mathrm{codim}V^s(\Omega).
\end{split}
\end{equation}
\end{lemma}

\begin{proof}
By Remark \ref{invari-space}, it is clear that $\dim V^c(M)\le 2$ and $\dim V^u(M)>0$.

If $\dim V^c(M)=1$, by Lemma \ref{hyperbolic-minimal}(2), $M$ is spatially-inhomogeneous and \eqref{dimension1} is established. Again, by the same arguments in \cite[Lemma 5.3]{SWZ2}, one can find a $\delta^*>0$ (independent the choose of $\omega\in M$) such that $M^c(\omega,\delta^*)\subset \Sigma u$ for any $\omega=(u,g)\in M$.

If $\dim V^c(M)=0$, then it follows from Lemma \ref{hyperbolic-minimal}(1) that $M$ is a spatially-homogeneous $1$-cover of $H(f)$. Moreover, by Lemma \ref{hyperbolic2}, $\dim V^u(M)$ must be odd. Therefore, we obtain \eqref{dimension2}.
\end{proof}

\begin{lemma}\label{constant-ontwoset}
 Assume that one of assumptions {\bf (H0)-(H2)} holds. Let $M_1, M_2\subset \Omega$ be two minimal sets with $\Sigma M_1\cap M_2=\emptyset$. Then, there exists an integer $N\in \mathbb{N}$ such that \begin{equation}\label{constant-0}
z(\varphi(t,\cdot;\sigma_{a_1}u_1,g)-\varphi(t,\cdot;\sigma_{a_2}u_2,g))=N,
\end{equation}
for any $t\in \mathbb{R}$, $g\in H(f)$, $(u_i,g)\in M_i$ and $a_i\in S^1$, $i=1,2$.
\end{lemma}
\begin{proof}
We only prove \eqref{constant-0} under the assumption of {\bf (H2)}, while for {\bf (H0)} or {\bf (H1)} the proof is similar. Note that $z(\varphi(t,\cdot;\sigma_{a_1}u_1,g)
-\varphi(t,\cdot;\sigma_{a_2}u_2,g))=z(\varphi(t,\cdot;\sigma_{(a_1-a_2)}u_1,g)-\varphi(t,\cdot;u_2,g))$. Then, in order to prove \eqref{constant-0}, it suffices to find some integer $N\in \mathbb{N}$ such that
\begin{equation}\label{constant-00}
z(\varphi(t,\cdot;\sigma_{a}u_1,g)-\varphi(t,\cdot;u_2,g))=N,
\end{equation}
for any $t\in \mathbb{R},g\in H(f),a\in S^1$ and $(u_i,g)\in M_i$, $i=1,2$.

To this end, we observe that, by Lemma \ref{homogeneous-2d}, $\dim V^c(M_i)\leq 2$ and $\dim V^u(M_i)>0$ ($i=1,2$). Then it follows from Lemma \ref{hyperbolic-minimal}(1) or Lemma \ref{normhyper-minimal}(i)-(ii) that, in any case, there exists (at least) a residual invariant set $Y_0\subset H(f)$ such that, for any $g\in Y_0$, there exist $u^i_g\in X$ ($i=1,2$) with $p^{-1}(g)\cap M_i\subset (\Sigma u^i_g,g)$.

 Now, for each $g\in H(f)$ and $(u_i,g)\in M_i\cap p^{-1}(g)$($i=1,2$), we {\it claim that there is an integer $N\in \mathbb {N}$ such that $z(\varphi(t,\cdot;\sigma_au_1,g)-\varphi(t,\cdot;u_2,g))=N$ for all $t\in \mathbb {R}$ and $a\in S^1$}. In order to prove this claim, for such $g$ and $(u_i,g)$, we first note that there are $T>0$ and $N_1, N_2$ such that
\begin{equation}\label{positive-constant}
 z(\varphi(t,\cdot;\sigma_a u_1,g)-\varphi(t,\cdot;u_2,g))=N_1 \quad \text{for all }t\ge T\text{ and } a\in S^1,
\end{equation}
and
\begin{equation}\label{negative-constant}
 z(\varphi(t,\cdot;\sigma_au_1,g)-\varphi(t,\cdot;u_2,g))=N_2 \quad \text{for all}\ t\le -T\text{ and } a\in S^1.
\end{equation}
In fact, since $u_2\notin \Sigma u_1$, \eqref{positive-constant} follows directly from Corollary \ref{difference-lapnumber}(a), the connectivity and compactness of $S^1$. As for \eqref{negative-constant}, one can take a sequence $t_n\to -\infty$ such that $\Pi^{t_n}(u_i,g)$($i=1,2$) converges to $(\tilde u_i,\tilde g)\in M_i\cap p^{-1}(\tilde g)$ as $n\to \infty$, for $i=1,2$. Recall that $\Sigma M_1\cap M_2=\emptyset$. Then by Lemma \ref{sequence-limit} and the connectivity of $S^1$,  there is an $N_2\in\mathbb{N}$ such that
\begin{equation}\label{limit-const}
  z(\varphi(t,\cdot;\sigma_a\tilde u_1,\tilde g)-\varphi(t,\cdot;\tilde u_2,\tilde g))=N_2,\quad\text{for all}\ a\in S^1,\ t\in \mathbb{R}.
\end{equation}
Therefore, for any $a\in S^1$, one has
\begin{equation}\label{asympto}
 z(\varphi(t_n,\cdot;\sigma_au_1,g)-\varphi(t_n,\cdot;u_2,g))=N_2,\quad
\end{equation}
for all $n$ sufficiently large. Hence, combined by Corollary \ref{difference-lapnumber}(a), the connectivity and compactness of $S^1$ again imply that \eqref{negative-constant} holds.

We now turn to prove that $N_1=N_2$. Choose a sequence $t_n\to \infty$ such that $\Pi^{t_n}(u_2,g)\to (u_2,g)$ as $n\to \infty$. Without loss of generality, we assume that $\Pi^{t_n}(u_1,g)\to (\bar u_1,g)$. By Lemma \ref{sequence-limit} again, there is an integer $N>0$ satisfying that
\begin{equation}\label{constant-1}
  z(\varphi(t,\cdot;\sigma_a\bar u_1,g)-\varphi(t,\cdot;u_2,g))=N,\quad\text{for all}\ a\in S^1,\ t\in \mathbb{R}.
\end{equation}
 Clearly, $(u_1,g)$, $(\bar u_1,g)\in M_1\cap p^{-1}(g)$. Choose some sequence $t^*_n\to\infty$ such that $\Pi^{t^*_n}(u_1,g)\to (u^*_1,g^*)\in M_1$ with $g^*\in Y_0$. By the property of $Y_0$ and the translation invariance, one may obtain that $\Pi^{t^*_n}(\sigma_{a_+}\bar u_1,g)\to (u^*_1,g^*)$ for some $a_+\in S^1$. Together with \eqref{positive-constant}, \eqref{constant-1} and the continuity of $z(\cdot)$, this then implies that $N=N_1$. Likewise, one can find $N=N_2$ by using \eqref{negative-constant}, \eqref{constant-1} and replacing $t^*_n$ by some similar sequence $s^*_n\to -\infty$. Therefore, one has $N_1=N=N_2$. Thus, we have proved the claim.

Finally, we show that $N$ is independent of $g\in H(f)$ and $(u_i,g)\in M_i\cap p^{-1}(g)$ ($i=1,2$). Indeed, for any $g\in H(f)$ and any $(u_i,g)$, $(\hat u_i,g)\in M_i\cap p^{-1}(g)$($i=1,2$), By the claim above, there are $N_1,N_2\in\mathbb{N}$ such that
\begin{equation*}
  z(\varphi(t,\cdot;\sigma_au_1,g)-\varphi(t,\cdot;u_2,g))=N_1,\quad \text{for all}\ a\in S^1, \ t\in \mathbb{R},
\end{equation*}
and
\begin{equation*}
  z(\varphi(t,\cdot;\sigma_a\hat u_1,g)-\varphi(t,\cdot;\hat u_2,g))=N_2,\quad \text{for all}\ a\in S^1,\ t\in \mathbb{R}.
\end{equation*}
Choose some $(u^*_i,g^*)\in M_i$ ($i=1,2$) with $g^*\in Y_0$. Then there are $t_n\to\infty$, $a_i,\hat a_i\in S^1$ such that $\Pi^{t_n}(\sigma_{a_i}u_i,g)\to (u^*_i,g^*)$ and $\Pi^{t_n}(\sigma_{\hat a_i}u_i,g)\to (u^*_i,g^*)$ as $n\to\infty$. The continuity of $z(\cdot)$ then implies that
\begin{eqnarray*}
\begin{split}
  N_1=z(\varphi(t,\cdot; \sigma_{a_1}u_1,g)-\varphi(t,\cdot;\sigma_{a_2}u_2,g))&=z(u^*_1-u^*_2)\\ &=z(\varphi(t,\cdot;\sigma_{\hat a_1}\hat u_1,g)-\varphi(t,\cdot;\sigma_{\hat a_2}\hat u_2,g))=N_2.
\end{split}
\end{eqnarray*}
Moreover, for any $g,\hat g\in H(f)$ and $(u_i,g)\in M_i\cap p^{-1}(g)$, $(\hat u_i,\hat g)\in M_i\cap p^{-1}(\hat g)$($i=1,2$). Again, one can choose a sequence  $t_n\to -\infty$ and $(\bar u_ 2, \hat g)\in M_2\cap p^{-1}(\hat g)$ such that $\Pi^{t_n}(u_1,g)\to (\hat u_1,\hat g)$ and $\Pi^{t_n}(u_2,g)\to (\bar u_2,\hat g)$ as $n\to \infty$. Similarly as the arguments in \eqref{limit-const}-\eqref{asympto}, we have
\begin{eqnarray*}
\begin{split}
  N=z(\varphi(t,\cdot; u_1,g)-\varphi(t,\cdot;u_2,g))=z(\varphi(t,\cdot;\hat u_1,\hat g)-\varphi(t,\cdot;\hat u_2,\hat g)),
\end{split}
\end{eqnarray*}
for all $t\in \mathbb{R}$. Thus, we have proved that $N$ is independent of $g\in H(f)$, $a\in S^1$ and $(u_i,g)\in M_i\cap p^{-1}(g)$($i=1,2$), which completes the proof of the lemma.
\end{proof}

\section{Skew-product semiflow on  the quotient space}\label{S:SP-S-quotient}

In this section, we  introduce the skew-product semiflow on the quotient space induced by the spatial-translation and present some basic properties.

For any $u\in X$, we define an equivalence relation on $X$ by declaiming  $u \sim v$ if and only if $u=\sigma_a v$ for some $a\in S^1$, and denoted by $[\cdot]$ for the same equivalence class. Then $\tilde{X}=X/\sim$ (the quotient space of $X$) is a metric space with $\tilde{d}_{\widetilde{X}}$ defined as $\tilde{d}_{\tilde{X}}([u],[v]):=d_H(\Sigma u,\Sigma v)$ for any $[u],[v]\in \tilde{X}$. Here $d_H(U,V)$ is the Hausdorff metric of the compact subsets $U,V$ in $X$, defined as $d_{H}(U,V)=\sup\{\sup_{u \in U} \inf_{v \in V} d_X(u,v),$ $\, \sup_{v \in V} \inf_{u \in U} d_X(u,v)\}$ with the metric $d_X(u,v)=||u-v||_{X}$. It is clear that $d_X$ satisfies the $S^1$-translation invariance, that is, $d_X(\sigma_a u,\sigma_a v)=d_X(u,v)$ for any $u,v\in X$, $a \in S^1$. Let $d_Y$ be the metric on $H(f)$, then one can induce a product metric $d$ on $X\times H(f)$ by setting $d((u_1,g_1),(u_2,g_2))=d_X(u_1,u_2)+d_Y(g_1,g_2)$ for any two points $(u_1,g_1),(u_2,g_2)\in X\times H(f)$. Hence, an induced metric $\tilde d$ on $\tilde{X}\times H(f)$ can be defined as $\tilde{d}(([u],g_1),([v],g_2))=\tilde{d}_{\tilde X}([u],[v])+d_Y(g_1,g_2)$. For any subset $K\subset X\times H(f)$, we write $\tilde{K}=\{([u],g)\in \tilde{X}\times H(f)|(u,g)\in K\}$.

Consider the induced mapping $\tilde{\Pi}^t$ ($t\geq 0$) on $\tilde{X}\times H(f)$ as
\begin{align}\label{induced-skepro-semiflow}
\begin{split}
\tilde{\Pi}^t:\tilde{X}\times H(f)&\rightarrow \tilde{X}\times H(f);\\
([u],g)&\rightarrow (\tilde{\varphi}(t,\cdot;[u],g),g\cdot t):=([\varphi(t,\cdot;u,g)],g\cdot t).
\end{split}
\end{align}
It follows from\cite[Lemma 3.10]{SWZ} that $\tilde{\Pi}^t$ is a skew-product semiflow on $\tilde{X}\times H(f)$. It is also not difficult to see that if $E\subset X\times H(f)$ is a connected and compact invariant set of $\Pi^t$, then $\tilde E=\{([u],g)|(u,g)\in E\}$ is also a connected and compact invariant set of $\tilde \Pi^t$.  Moreover, other notations and definitions for $\tilde \Pi^t$ are analogous to those of $\Pi^t$, such as the (almost) $1$-cover property with respect to $\tilde \Pi^t$, the natural flow homomorphism $\tilde{p}:\tilde{X}\times H(f)\to H(f)$, etc.


Henceforth, we always write
\begin{equation}\label{E:Induced-Omega}
\tilde \Omega=\{([u],g)|(u,g)\in \Omega\},
\end{equation} whenever $\Omega=\omega(u_0,g_0)$ is an $\omega$-limit set of \eqref{equation-lim2}. Then the following lemma reveals that $\tilde{\O}$ is in fact the $\omega$-limit set of $([u_0],g)$ with respect to $\tilde{\Pi}^{t}$.

\begin{lemma}\label{induced-property2}
  Assume that $\Pi^{t}(u_0,g_0)$ is bounded for $t\geq 0$ and $\Omega=\omega(u_0,g_0)$ is the $\omega$-limit set of \eqref{equation-lim2}. Then $\tilde \Omega=\omega([u_0],g_0)$, where
$$\omega([u_0],g_0)=\{([u],g)\, |\, \text{ there exists }t_n\to \infty \text{ such that }\tilde\Pi^{t_n}([u_0],g_0)\to ([u],g)\}.
$$
\end{lemma}

\begin{proof}
  For any point $\tilde \omega\in \tilde \Omega$, there is $(u,g)\in \Omega$ such that $\tilde \omega=([u],g)$. Since $(u,g)\in \Omega$, there exists $t_n\to \infty$ such that $\Pi^{t_n}(u_0,g_0)\to (u,g)$ as $n\to \infty$. Then
\begin{align}
  \tilde d(([\varphi(t_n,\cdot;u_0,g_0)],g_0),([u],g))&=\tilde d_{\tilde X}([\varphi(t_n,\cdot;u_0,g_0)],[u])+d_{Y}(g_0\cdot t_n,g)\nonumber\\
  &=d_{H}(\Sigma\varphi(t_n,\cdot;u_0,g_0),\Sigma u)+d_{Y}(g_0\cdot t_n,g)\nonumber\\
  &\leq d_{X}(\varphi(t_n,\cdot;u_0,g_0),u)+d_{Y}(g_0\cdot t_n,g)\to 0,
\end{align}
which means that $\tilde \Pi^{t_n}([u_0],g_0)\to ([u],g)$ as $n\to \infty$. So, $\tilde \Omega\subset\omega([u_0],g_0)$.

On the other hand, given $\tilde \omega\in \omega([u_0],g_0)$, there exists $(u,g)\in X\times H(f)$ satisfies $\tilde \omega=([u],g)$. Since $\tilde \omega\in \omega([u_0],g_0)$, we assume $\tilde \Pi^{t_n}([u_0],g_0)\to ([u],g)$ for some $t_n\to\infty$. Then, there are $(u_g,g)\in \omega(u_0,g_0)$ and $\{t_{n_k}\}\subset \{t_n\}$ such that $\Pi^{t_{n_k}}(u_0,g_0)\to (u_g,g)$. By the arguments in the above paragraph, one can further to get $\tilde \Pi^{t_{n_k}}([u_0],g_0)\to ([u_g],g)$. Therefore, $([u],g)=([u_g],g)\in \tilde \Omega$, which entails that $\omega([u_0],g_0)\subset \tilde \Omega$. The proof of this lemma is completed.
\end{proof}

An immediate corollary of Lemma \ref{induced-property2} is the following
\begin{corollary}\label{induced-property3}
Let $M\subset X\times H(f)$ be a  minimal set of $\Pi^t$, then $\tilde M=\{([u],g)|(u,g)\in M\}$ is a minimal set of $\tilde\Pi^t$. Conversely, if $\tilde M$ ($\tilde M\subset \tilde \Omega$) is a minimal set of $\tilde\Pi^t$, then there is a minimal set $M\subset X\times H(f)$ ($M\subset \Omega$) such that $\tilde M=\{([u],g)|(u,g)\in M\}$.
\end{corollary}

\begin{lemma}\label{induced1-1cover}
Let $\Omega$ be an $\omega$-limit set of \eqref{equation-lim2} satisfying  one of
the hypotheses {\bf (H0)}-{\bf (H2)}.  Then we have
\begin{itemize}
\item[{\rm(i)}]Any minimal set $\tilde M\subset \tilde \O$ is an almost $1$-cover of $H(f)$. Moreover, if {\bf (H0)} holds, or {\bf (H1)} holds with $\dim V^u(\Omega)>0$, then $\tilde M$ is a $1$-cover of $H(f)$.
\item[{\rm(ii)}]Let $\tilde M_1$, $\tilde M_2\subset \tilde \O$ be two minimal sets of $\tilde\Pi^t$ and $M_1,M_2\subset \Omega$ be two minimal sets of $\Pi^t$ such that $\tilde M_i=\{([u_i],g)|(u_i,g)\in M_i\}$ ($i=1,2$). Define
\begin{eqnarray}\label{E:min-max}
 \begin{split}
   m_i(g):=\min\{m(u_i)|(u_i,g)\in M_i\cap p^{-1}(g)\},\\
   M_i(g):=\max\{m(u_i)|(u_i,g)\in M_i\cap p^{-1}(g)\}
 \end{split}
 \end{eqnarray}
 for $i=1,2$. Then $\tilde M_1$, $\tilde M_2$ are separated in the following sense:
 \begin{itemize}
   \item[{ \rm (ii-a)}] $[m_1(g),M_1(g)]\cap[m_2(g),M_2(g)]=\emptyset$ for all $g\in H(f)$;
   \item[{ \rm (ii-b)}] If $m_2(\tilde g)>M_1(\tilde g)$ for some $\tilde g\in H(f)$, then there exists $\delta>0$ such that  $m_2(g)>M_1(g)+\delta$ for all $g\in H(f)$.
\end{itemize}
\end{itemize}
\end{lemma}

\begin{proof}
(i) Let $\tilde M\subset\tilde \O$ be a minimal set of $\tilde \Pi^{t}$. Then by Corollary \ref{induced-property3}, there is a minimal set $M\subset \O$ such that $\tilde M=\{([u],g)|(u,g)\in M\}$. If {\bf (H0)} is satisfied, then by Remark \ref{invari-space}, $M$ is hyperbolic. Hence, Lemma \ref{hyperbolic-minimal}(1) implies that any hyperbolic $M$ is a spatially-homogeneous $1$-cover of $H(f)$. If {\bf (H1)} holds and $\dim V^u(M)>0$, then it follows from Lemma \ref{homogeneous} and Lemma \ref{normhyper-minimal}(ii) that $\tilde M$ is $1$-cover of $H(f)$. If {\bf (H1)} holds and $\dim V^u(M)=0$, by Lemma \ref{hyperbolic-minimal}, $M$ is at least a spatially-homogeneous almost $1$-cover of $H(f)$. Hence, $\tilde{M}$ is an almost $1$-cover of $H(f)$. Finally, if {\bf (H2)} holds, then by Remark \ref{invari-space}, $\dim V^c(M)\leq 2$. When $\dim V^c(M)=2$, Lemma \ref{normhyper-minimal}(i) directly entails that $\tilde M$ is an almost $1$-cover of $H(f)$. For other cases, one can combine  Lemma \ref{homogeneous-2d} and the similar arguments as above to obtain that $\tilde M$ is a $1$-cover of $H(f)$.

  (ii-a) Suppose on the contrary that there exists some $g\in H(f)$ such that $m_1(g)\le M_2(g)$ and $m_2(g)\le M_1(g)$. On the one hand, we choose $(u_1,g)\in M_1$ such that $m(u_1)=m_1(g)$, and choose $(u_2,g)\in M_2$ such that $m(u_2)=M_2(g)$. So, $m(u_1)=m_1(g)\le M_2(g)=m(u_2)$. Recall that $\Sigma M_1\cap M_2=\emptyset$ (since $\tilde{M}_1\ne \tilde{M_2}$). Then Lemma \ref{constant-ontwoset} implies that
  \begin{equation}
z(\varphi(t,\cdot;\sigma_{a_1}u_1,g)-\varphi(t,\cdot;\sigma_{a_2}u_2,g))\equiv {\rm constant},
\end{equation}
for all $a_1,a_2\in S^1$ and $t\in \mathbb{R}$.
By virtue of Corollary  \ref{difference-lapnumber}, $\varphi(t,\cdot;\sigma_{a_1}u_1,g)-\varphi(t,\cdot;\sigma_{a_2}u_2,g)$ has only simple zeros on $S^1$, which entails that $m(\varphi(t,\cdot;u_1,g))\neq m(\varphi(t,\cdot;u_2,g))$ for any $t\in \mathbb{R}$.
Together with $m(u_1)\le m(u_2)$, one obtains that $m(u_1)<m(u_2)$; and hence,
\begin{equation}\label{smaller-1}
  m(\varphi(t,\cdot;u_1,g))<m(\varphi(t,\cdot;u_2,g)),\quad \text{for all}\ t\in\mathbb{R}.
\end{equation}
By the minimality of $M_1$, one can find a sequence $t_n\to \infty$ such that $\Pi^{t_n}(u_1,g)\to (u^*_1,g^*)$ as $n\to \infty$, where $(u^*_1,g^*)\in M_1$ with $m(u^*_1)=M_1(g^*)$. For simplicity, we may also assume that $\Pi^{t_n}(u_2,g)\to (u^*_2,g^*)$ as $n\to \infty$. By \eqref{smaller-1}, $M_1(g^*)=m(u^*_1)\leq m(u^*_2)\le M_2(g^*)$. Moreover, it follows from Lemma \ref{constant-ontwoset} again that $m(u^*_1)\ne m(u^*_2)$, which means that $M_1(g^*)<M_2(g^*)$. On the other hand, together with  $m_2(g)\le M_1(g)$, one can repeat the similar argument above to obtain that $M_2(g^*)<M_1(g^*)$ Thus, we have obtained a contradiction; and hence, we have proved that for any $g\in H(f)$, either $m_1(g)>M_2(g)$ or $m_2(g)> M_1(g)$, which implies (ii-a) directly.

(ii-b)  We first show that if $m_2(\tilde g)>M_1(\tilde g)$ for some $\tilde g\in H(f)$, then $m_2(g)>M_1(g)$ for all $g\in H(f)$. Suppose that there is a $g^*\in H(f)$ such that $ m_2(g^*)\le M_1(g^*)$. Choose $(u^*_2,g^*)\in M_2$ with $m(u^*_2)=m_2(g^*)$, and choose $(u^*_1,g^*)\in M_1$ with $m(u^*_1)=M_1(g^*)$. Hence, we have $m(u^*_2)\le m(u^*_1)$. By the minimality of $M_2$, one can find a sequence $t_n\to \infty$ such that $\Pi^{t_n}(u^*_2,g^*)\to (u^{**}_2,\tilde g)$ as $n\to \infty$ with $m(u^{**}_2)=m_2(\tilde g)$. Without loss of generality, one may also assume that $\Pi^{t_n}(u^*_1,g^*)\to (u^{**}_1,\tilde g)$ as $n\to \infty$. By repeating the same arguments in the previous paragraph, one has
\begin{equation*}
  m_2(\tilde g)=m(u^{**}_2)\le m(u^{**}_1)\le M_1(\tilde g),
\end{equation*}
contradicting our assumption. Therefore, $m_2(g)>M_1(g)$ for all $g\in H(f)$.

Finally, we show the existence of $\delta>0$. Suppose that there is a sequence $\{g_n\}\subset H(f)$ such that $m_2(g_n)>M_1(g_n)$ and $\abs{m_2(g_n)-M_1(g_n)}\to 0$ as $n\to \infty$. Without loss of generality, let $g_n\to g^*\in H(f)$, $m_2(g_n)\to c$ and $M_1(g_n)\to c$ as $n\to\infty$, for some $c\in\mathbb{R}$. Since $M_i$ ($i=1,2$) are compact, $c\in[m_1(g^*),M_1(g^*)]\cap[m_2(g^*),M_2(g^*)]$, a contradiction to (ii-a).
\end{proof}

\begin{lemma}\label{proximal-equivalent}
For any two points $([u_1],g),([u_2],g)\in \tilde \O$ ($(u_i,g)\in \O$, $i=1,2$), if there exists $t_n\to\infty$ {\rm (resp. $s_n\to-\infty$)} such that $\tilde\Pi^{t_n}([u_1],g)-\tilde\Pi^{t_n}([u_2],g)\to 0$ {\rm(resp. $\tilde\Pi^{s_n}([u_1],g)-\tilde\Pi^{s_n}([u_2],g)\to 0$)} as $n\to \infty$. Then there exist a subsequence $\{t_{n_k}\}\subset \{t_{n}\}$ {\rm(resp. $\{s_{n_k}\}\subset \{s_{n}\}$)}, $a^*\in S^1$ and $(u^*,g^*)\in \Omega$ such that
\begin{equation*}
\begin{split}
&\Pi^{t_{n_k}}(u_1,g)\to (u^*,g^*)\quad {\rm and}\quad  \Pi^{t_{n_k}}(\sigma_{a^*}u_2,g)\to (u^*,g^*)\\
&(\mathrm{resp}.\ \Pi^{s_{n_k}}(u_1,g)\to (u^*,g^*)\quad {\rm and}\quad  \Pi^{s_{n_k}}(\sigma_{a^*}u_2,g)\to (u^*,g^*)).
\end{split}
\end{equation*}
\end{lemma}

\begin{proof}
 We only prove the case that $t_n\to \infty$, while the case that $s_n\to -\infty$ is similar. By the definition of metric on $\tilde X\times H(f)$, it then follows from $\tilde\Pi^{t_n}([u_1],g)-\tilde\Pi^{t_n}([u_2],g)\to 0$  that there exists $a_i^n\in S^1$ ($i=1,2$) such that
\begin{equation}\label{proximal-seuqen}
   \Pi^{t_{n}}(\sigma_{a_1^n}u_1,g)-\Pi^{t_{n}}(\sigma_{a_2^n}u_2,g)\to 0.
\end{equation}
Since both $\O$ and $S^1$ are compact, one may assume
\begin{equation}\label{E:u-i-a-i-conver}
a_i^n\to a_i^*\in S^1 \text{ and }\Pi^{t_n}(u_i,g)\to (u_i^*,g^*)\in \Omega
\end{equation} as $n\to\infty$, for $i=1,2$. Recall also that
\begin{equation*}
\begin{split}
\|\varphi(t_{n},\cdot;\sigma_{a^n_i} u_i,g)-\sigma_{a^*_i}u_i^*\|&\le \|\varphi(t_{n},\cdot;\sigma_{a^n_i} u_i,g)-\varphi(t_{n},\cdot;\sigma_{a^*_i} u_i,g)\|+\|\varphi(t_{n},\cdot;\sigma_{a^*_i} u_i,g)-\sigma_{a^*_i}u_i^*\|\\
&= \|\sigma_{a^n_i}\varphi(t_{n},\cdot; u_i,g)-\sigma_{a^*_i}\varphi(t_{n},\cdot; u_i,g)\|+\|\sigma_{a^*_i} \varphi(t_{n},\cdot;u_i,g)-\sigma_{a^*_i}u_i^*\|\\
&= \|\sigma_{a^n_i-a^*_i}\varphi(t_{n},\cdot; u_i,g)-\varphi(t_{n},\cdot; u_i,g)\|+\| \varphi(t_{n},\cdot;u_i,g)-u_i^*\|,
\end{split}
\end{equation*}
where the last two equalities are due to the translation invariance of the semiflow and
the metric $d_X(\cdot,\cdot)$, respectively. Together with \eqref{E:u-i-a-i-conver} and the compactness of $\Omega$, this implies that $\|\varphi(t_{n},\cdot;\sigma_{a^n_i} u_i,g)-\sigma_{a^*_i}u_i^*\|\to 0$ as $n\to \infty$, that is, $\Pi^{t_n}(\sigma_{a^n_i}u_i,g)\to (\sigma_{a^*_i}u_i^*,g^*)$ as $n\to \infty$, for $i=1,2$. Combing with \eqref{proximal-seuqen}, one has $\sigma_{a^*_1}u^*_1=\sigma_{a^*_2}u^*_2$. Let $u^*=u_1^*$ and $a^*=a^*_1-a^*_2$, then we have $\Pi^{t_n}(u_1,g)\to (u^*,g^*)$ and $\Pi^{t_n}(u_2,g)\to (\sigma_{a^*}u^*,g^*)$ as $n\to\infty$. The proof of this lemma is completed.
\end{proof}
\begin{lemma}\label{inducedproxi}
    Assume that {\bf (H1)} holds  and $\dim V^u(\O)>0$. Let $\tilde M\subset \tilde\Omega$ be any minimal set of $\tilde{\Pi}^t$. Then for any $([u_1],g)\in \tilde M$ and $([u_2],g)\in \tilde \Omega\setminus \tilde M$, $\{([u_1],g),([u_2],g)\}$ can not be two sided proximal pair.
\end{lemma}
\begin{proof}
 By Lemma \ref{induced1-1cover}(i), $\tilde M$ is $1$-cover of $H(f)$;  and moreover, Corollary \ref{induced-property3} implies that there exists a minimal set $M\subset \O$ such that $\tilde M=\{([u],g)|(u,g)\in M\}$.  Suppose that there are $([u_1],g)\in \tilde M$ and $([u_2],g)\in \tilde \O\setminus \tilde M$ (hence, one has $(u_1,g)\in M$ and $(u_2,g)\in \O\setminus \Sigma M$) such that $\{([u_1],g),([u_2],g)\}$ forms a two sided proximal pair. By virtue of Lemma \ref{proximal-equivalent}, there are $a^*,a^{**}\in S^1$, as well as two sequences $t_n\to \infty$ and $s_n\to -\infty$, such that
\begin{equation}\label{posi-proxi2}
\begin{split}
 \Pi^{t_n}(u_1,g)\to(u^*,g^*) \,\,\text{ and }\,\,
 \Pi^{t_n}(\sigma_{a^*}u_2,g)\to(u^*,g^*)
\end{split}
\end{equation}
and
\begin{equation}\label{nega-proxi2}
\begin{split}
 \Pi^{s_n}(u_1,g)\to(u^{**},g^{**}) \,\,\text{ and }\,\,
 \Pi^{s_n}(\sigma_{a^{**}}u_2,g)\to(u^{**},g^{**})
\end{split}
\end{equation}
as $n\to\infty$, where $(u^*,g^*), (u^{**},g^{**})\in M$.

Since {\bf (H1)} holds and  $\dim V^u(\O)>0$, the minimal set $M$ satisfies one of the cases (a)-(b) in Lemma \ref{homogeneous}. In the following, we will show that both of these two cases lead to certain contradiction, respectively. Based on this, one can conclude that $\{([u_1],g),([u_2],g)\}$ is not two sided proximal pair.

Case (i). If $M$ satisfies (b) in Lemma \ref{homogeneous}, then $M$ is a spatially-homogeneous $1$-cover of $H(f)$. In particular, $u^*,u^{**}$ are spatially-homogeneous. So, \eqref{posi-proxi2} and \eqref{nega-proxi2} turn out to be
\begin{equation*}\label{posi-proxi3}
\begin{split}
 \Pi^{t_n}(u_1,g)\to(u^*,g^*) \,\,\text{ and }\,\,
\Pi^{t_n}(u_2,g)\to(u^*,g^*)
\end{split}
\end{equation*}
and
\begin{equation*}\label{nega-proxi3}
\begin{split}
 \Pi^{s_n}(u_1,g)\to(u^{**},g^{**}) \,\,\text{ and }\,\,
 \Pi^{s_n}(u_2,g)\to(u^{**},g^{**}).
\end{split}
\end{equation*}
Hence, $\{(u_1,g),(u_2,g)\}$ is a two sided proximal pair, contradicting to Lemma \ref{proximalcons}.

Case (ii). If $M$ satisfies (a) in Lemma \ref{homogeneous}, then we {\it claim that
\begin{equation}\label{constant-cent1}
 z(\varphi(t,\cdot;\sigma_au_1,g)-\varphi(t,\cdot;u_2,g))=N_u\quad \text{for all }t\in \mathbb{R}\,
 \text{and }a\in S^1,
\end{equation}
where $N_u$ is defined in \eqref{E:Nu-def}.} Before giving the proof of this claim, we will first show how this claim induces certain contradiction.

In fact, by virtue of Lemma \ref{zero-cons-local} and the compactness of $S^1$, the claim \eqref{constant-cent1} implies that there exists $\delta>0$ (independent of $a\in S^1$) such that
\begin{equation}\label{case-A-const-1}
z(u_2-\sigma_a u_1+v)=N_u,\quad \text{for any $a\in S^1$ and $\|v\|<\delta$.}
\end{equation}
When $\dim V^u(\Omega)$ is even (resp. $\dim V^u(\Omega)$ is odd), we let $a_0=2\pi-a^*$ (resp. $a_0=2\pi-a^{**}$). Then, together with \eqref{posi-proxi2} (resp. \eqref{nega-proxi2}), Lemma \ref{hyperbolic1}(ii) implies there exists $v_n\in M^u(\varphi(t_n,\cdot;u_2,g),g\cdot t_n,\delta^*)\cap M^{cs}(\varphi(t_n,\cdot;\sigma_{a_0}u_1,g),g\cdot t_n,\delta^*)$ (resp. $v_n\in M^s(\varphi(s_n,\cdot;u_2,g),g\cdot s_n,\delta^*)\cap M^{cu}(\varphi(s_n,\cdot;\sigma_{a_0}u_1,g),g\cdot s_n,\delta^*)$) for all $n$ sufficiently large.

We now {\it assert} that $v_n\notin M^c(\varphi(t_n,\cdot;\sigma_{a_0}u_1,g),g\cdot t_n,\delta^*)$ (resp. $v_n\notin M^c(\varphi(s_n,\cdot;\sigma_{a_0}u_1,g),g\cdot s_n,\delta^*)$). Indeed, suppose not, then one can replace $M$ by $\sigma_{a_0} M$  in Lemma \ref{homogeneous}(a) (because of the minimality of $\sigma_{a_0} M$ and $\dim V^c(\sigma_{a_0}M)=\dim V^c(M)=1$), and obtains that $v_n=\sigma_{a_n} \varphi(t_n,\cdot;\sigma_{a_0}u_1,g)$ (resp. $v_n=\sigma_{a_n} \varphi(s_n,\cdot;\sigma_{a_0}u_1,g)$) for some $a_n\in S^1$. Observe that $v_n\in M^u(\varphi(t_n,\cdot;u_2,g),g\cdot t_n,\delta^*)$ (resp. $M^s(\varphi(s_n,\cdot;u_2,g),g\cdot s_n,\delta^*)$), one has
\begin{equation}\label{backward-contracting-1}
\begin{split}
&\|\sigma_{a_n+a_0}u_1-u_2\|=\norm{\varphi(-t_n,\cdot;v_n,g\cdot t_n)-u_2}
\leq Ce^{-\frac{\alpha}{2} t_n}\|v_n-\varphi(t_n,\cdot;u_2,g)\|\le C\delta^*e^{-\frac{\alpha}{2} t_n}.\\
&(\mathrm{resp}.\ \|\sigma_{a_n+a_0}u_1-u_2\|=\norm{\varphi(-s_n,\cdot;v_n,g\cdot s_n)-u_2}
\leq Ce^{\frac{\alpha}{2} s_n}\|v_n-\varphi(s_n,\cdot;u_2,g)\|\le C\delta^*e^{\frac{\alpha}{2}s_n}).
\end{split}
\end{equation}
Since $u_{2}\notin \Sigma u_1$, $\varepsilon_0:=\inf_{a\in S^1}\|\sigma_{a}u_{1}-u_{2}\|>0$. But, by letting $n$ large enough in \eqref{backward-contracting-1}, one can obtain that $\|\sigma_{a_n+a_0}u_1-u_2\|<\varepsilon_0/2$, a contradiction. So, we have proved $v_n\notin M^c(\varphi(t_n,\cdot;\sigma_{a_0}u_1,g),\delta^*)$ (resp. $v_n\notin M^c(\varphi(s_n,\cdot;\sigma_{a_0}u_1,g),\delta^*)$).

Recall that $v_n\in M^{cs}(\varphi(t_n,\cdot;\sigma_{a_0} u_1,g),\delta^*)$ (resp. $v_n\in M^{cu}(\varphi(s_n,\cdot;\sigma_{a_0} u_1,g),\delta^*)$). By Remark \ref{stable-leaf}(4) and Lemma \ref{homogeneous}(a), there is some $\tilde a_n\in S^1$ such that $v_n\in M^{s}(\sigma_{\tilde a_n}\varphi(t_n,\cdot;\sigma_{a_0}u_1,g),\delta^*)$ (resp. $v_n\in M^{u}(\sigma_{\tilde a_n}\varphi(s_n,\cdot;\sigma_{a_0}u_1,g),\delta^*)$) with $\sigma_{\tilde a_n}\varphi(t_n,\cdot;\sigma_{a_0}u_1,g)\in M^c(\varphi(t_n,\cdot;\sigma_{a_0}u_1,g),\delta^*)$ (resp. $\sigma_{\tilde a_n}\varphi(s_n,\cdot;\sigma_{a_0}u_1,g)\in M^c(\varphi(s_n,\cdot;\sigma_{a_0}u_1,g),\delta^*)$) for $n$ sufficiently large. Recall that $\dim V^u(\Omega)$ is even (resp. $\dim V^u(\Omega)$ is odd), Lemma \ref{zerocenter}(3) (resp. Lemma \ref{zerocenter}(2)) entails that $z(v_n-\sigma_{\tilde a_n}\varphi(t_n,\cdot;\sigma_{a_0}u_1,g))\ge N_u+2$ (resp. $z(v_n-\sigma_{\tilde a_n}\varphi(s_n,\cdot;\sigma_{a_0}u_1,g))\le N_u-2$), for $n$ sufficiently large. Therefore, by Corollary \ref{difference-lapnumber}(a)
\begin{equation}\label{u-3-Nu-large}
z(\varphi(-t_n,\cdot;v_n,g\cdot t_n)-\sigma_{\tilde a_n+a_0}u_1)\geq N_u+2, \,\,\,\,(\mathrm{resp}.\ z(\varphi(-s_n,\cdot;v_n,g\cdot s_n)-\sigma_{\tilde a_n+a_0}u_1)\leq N_u-2),
\end{equation}
for $n$ sufficiently large.
On the other hand, \eqref{backward-contracting-1} implies that $\|\varphi(-t_n,\cdot;v_n,g\cdot t_n)-u_2\|<\delta$ (resp. $\|\varphi(-s_n,\cdot;v_n,g\cdot s_n)-u_2\|<\delta$), for $n$ sufficiently large, where $\delta>0$ is as defined in \eqref{case-A-const-1}. Therefore, by using \eqref{case-A-const-1}, one has
\begin{equation*}
\begin{split}
z(\varphi(-t_n,\cdot;v_n,g\cdot t_n)-\sigma_{\tilde a_n+a_0}u_1)&=z(\varphi(-t_n,\cdot;v_n,g\cdot t_n)-u_2+u_2-\sigma_{\tilde a_n+a_0}u_1)\\
&=z(u_2-\sigma_{\tilde a_n+a_0}u_1)=N_u,\\
(\mathrm{resp}.\ z(\varphi(-s_n,\cdot;v_n,g\cdot s_n)-\sigma_{\tilde a_n+a_0}u_1)&=z(\varphi(-s_n,\cdot;v_n,g\cdot s_n)-u_2+u_2-\sigma_{\tilde a_n+a_0}u_1)\\
&=z(u_2-\sigma_{\tilde a_n+a_0}u_1)=N_u),
\end{split}
\end{equation*}
for $n\gg 1$. Consequently, we have obtained a contradiction to \eqref{u-3-Nu-large}. Thus, based on the claim \eqref{constant-cent1}, we have obtained certain ``contradiction" for Case (ii). Therefore, as we mentioned above, this implies that $\{([u_1],g),([u_2],g)\}$ can not be two sided proximal pair.
\vskip 1mm

Finally, it remains to prove the claim \eqref{constant-cent1}. Indeed, given any $a\in S^1$ with $\sigma_a u^*\neq u^*$,  Lemma \ref{normhyper-minimal}(iii) means that $\sigma_a u^*-u^*$ has only simple zeros and $z(\sigma_a u^*-u^*)=N_u$. Thus, by Lemma \ref{zero-cons-local} and \eqref{posi-proxi2}, one has
\begin{equation*}
  z(\varphi(t_n,\cdot;\sigma_a u_1,g)-\varphi(t_n,\cdot;\sigma_{a^*}u_2,g))=z(\sigma_a u^*-u^*)=N_u,
\end{equation*}
for $n$ sufficiently large. So, Corollary \ref{difference-lapnumber}(c) immediately reveals that, for any  $a\in S^1$ with $\sigma_a u^*\neq u^*$, there is $T_a\in \mathbb{R}$ such that
\begin{equation}\label{constant-cent}
  z(\varphi(t,\cdot;\sigma_a u_1,g)-\varphi(t,\cdot;\sigma_{a^*}u_2,g))=N_u, \quad t\geq T_a.
\end{equation}
Meanwhile, we also need to consider the element $a_0\in S^1$ with $\sigma_{a_0} u^*= u^*$. For such $a_0\in S^1$, Corollary \ref{difference-lapnumber}(c) implies there are $N_0\in\mathbb{N}$ and $T_0>0$ such that
\begin{equation*}
  z(\varphi(t,\cdot;\sigma_{a_0}u_1,g)-\varphi(t,\cdot;\sigma_{a^{*}}u_2,g))=N_0,
\end{equation*}
for all $t\geq T_0$. So, by Lemma \ref{zero-cons-local}, there is $\delta_0>0$ such that for any $a\in S^1$ with $|a-a_0|<\delta_0$, one has $z(\varphi(T_0,\cdot;\sigma_a u_1,g)-\varphi(T_0,\cdot;\sigma_{a^*}u_2,g))=N_0$. Recall that $u^*$ is spatially-inhomogeneous in Lemma \ref{homogeneous}(a). Then there also exists $\tilde a\in S^1$ with $|\tilde a-a_0|<\delta$ satisfies $\sigma_{\tilde a}u^*\neq u^*$ and $z(\varphi(T_0,\cdot;\sigma_{\tilde a} u_1,g)-\varphi(T_0,\cdot;\sigma_{a^{*}}u_2,g))=N_0$. By \eqref{constant-cent} and Corollary \ref{difference-lapnumber}, one has $N_0\geq N_u$. Thus, it follows that
\begin{equation*}
 z(\varphi(t,\cdot;\sigma_au_1,g)-\varphi(t,\cdot;\sigma_{a^*}u_2,g))\ge N_u\quad \text{ for all }a\in S^1\text{and } t\in \mathbb{R},
\end{equation*}
or equivalently,
\begin{equation}\label{geq-estimate}
 z(\varphi(t,\cdot;\sigma_au_1,g)-\varphi(t,\cdot;u_2,g))\ge N_u\quad \text{ for all }a\in S^1\text{and } t\in \mathbb{R}.
\end{equation}

By repeating the similar deduction under the situation \eqref{nega-proxi2}, one can also obtain that
\begin{equation}\label{leq-estimate}
 z(\varphi(t,\cdot;\sigma_au_1,g)-\varphi(t,\cdot;\sigma_{a^{**}}u_2,g))\le N_u, \quad t\in\mathbb{R},
\end{equation}
for all $a\in S^1$ with $\sigma_{a} u^{**}\neq u^{**}$. Meanwhile, for  the element $a_1\in S^1$ with $\sigma_{a_1} u^{**}=u^{**}$, we need to consider two subcases:
\begin{equation*}
\begin{split}
  & \text{(Sub-I):}\ \|\varphi(t,\cdot;\sigma_{a_1} u_1,g)-\varphi(t,\cdot;\sigma_{a^{**}}u_2,g)\|\to 0 \,\,(\text{as }t\to -\infty), \ \text{and}\\
  & \text{(Sub-II):}\ \|\varphi(t,\cdot;\sigma_{a_1} u_1,g)-\varphi(t,\cdot;\sigma_{a^{**}}u_2,g)\|\nrightarrow 0\,\,(\text{as }t\to -\infty),\ \text{respectively}.
\end{split}
\end{equation*}

When (Sub-I) holds, by Remark \ref{stable-leaf}(3), we have $\varphi(t,\cdot;\sigma_{a^{**}}u_2,g)\in M^{cu}(\Pi^{t}(\sigma_{a_1} u_1,g),\delta^*)$ for $t$ sufficiently negative. Thus, by Lemma \ref{zerocenter} (2) or (3) (depending on whether $\dim V^u(\Omega)$ is odd or even), there is $T>0$ such that
\begin{equation}\label{leq-estimate1}
  z(\varphi(t,\cdot;\sigma_{a_1}u_1,g)-\varphi(t,\cdot;\sigma_{a^{**}}u_2,g))\le N_u,
\end{equation}
for all $t<-T$.

When (Sub-II) holds, there exist $l_n\to -\infty$ and two distinct points $(\tilde u_1,\tilde g)\in \sigma_{a_1} M_1$, $(\tilde u_2,\tilde g)\in \sigma_{a^{**}} \Omega$, such that $\Pi^{l_n}(\sigma_{a_1} u_1,g) \to (\tilde u_1,\tilde g)$ and $\Pi^{l_n}(\sigma_{a^{**}} u_2,g) \to (\tilde u_2,\tilde g)$ as $n\to\infty$. By Lemma \ref{sequence-limit}, $\tilde u_1-\tilde u_2$ has only simple zeros on $S^1$. Let $N_1=z(\tilde u_1-\tilde u_2)$, Then Lemma \ref{zero-cons-local} implies that
\begin{equation*}
  z(\varphi(l_n,\cdot;\sigma_{a_1}u_1,g)-\varphi(l_n,\cdot;\sigma_{a^{**}}u_2,g))=N_1,
\end{equation*}
for $n$ sufficiently large. So, again by Corollary \ref{difference-lapnumber}(a), there is $T_1\in \mathbb{R}$ such that
\begin{equation*}
  z(\varphi(t,\cdot;\sigma_{a_1}u_1,g)-\varphi(t,\cdot;\sigma_{a^{**}}u_2,g))=N_1,
\end{equation*}
for all $t\leq T_1$. Choose some $\delta_1>0$ such that for any $a\in S^1$ with $|a-a_1|<\delta_0$, one has $z(\varphi(T_1,\cdot;\sigma_a u_1,g)-\varphi(T_1,\cdot;\sigma_{a^{**}}u_2,g))=N_1$. Noticing again that $u^{**}$ is spatially-inhomogeneous, there also exists $\tilde a\in S^1$ with $|\tilde a-a_1|<\delta_1$ satisfies $\sigma_{\tilde a}u^{**}\neq u^{**}$ and $z(\varphi(T_1,\cdot;\sigma_{\tilde a} u_1,g)-\varphi(T_1,\cdot;\sigma_{a^{**}_2}u_2,g))=N_1$. So, by \eqref{leq-estimate}, one has $N_1\leq N_u$. Thus, we also obtain \eqref{leq-estimate1} for subcase (Sub-II). Therefore, from \eqref{leq-estimate1}, we have
\begin{equation*}
 z(\varphi(t,\cdot;\sigma_au_1,g)-\varphi(t,\cdot;\sigma_{a^{**}}u_2,g))\le N_u,\quad \text{ for all }a\in S^1,\ t\in \mathbb{R}.
\end{equation*}
In other words,
\begin{equation}\label{leq-estimate2}
 z(\varphi(t,\cdot;\sigma_au_1,g)-\varphi(t,\cdot;u_2,g))\le N_u\quad \text{ for all }a\in S^1\text{and } t\in \mathbb{R}.
\end{equation}
Combing \eqref{leq-estimate2} with \eqref{geq-estimate}, we have proved the claim \eqref{constant-cent1}. The proof of Lemma \ref{inducedproxi} is completed.
\end{proof}

\section{Structure of $\omega$-limit set $\O$}

In this section, we will investigate the structure of the $\omega$-limit set $\O:=\omega(u_0,g_0)$ of any bounded positive orbit of $\Pi^{t}(u_0,g_0)$ for \eqref{equation-lim2}.
We first state three main Theorems of this paper, followed by the proofs of these theorems in three separated subsections.

\begin{theorem}\label{hyperbolic0}
    Assume that the $\omega$-limit set $\Omega$ satisfies {\bf (H0)}. Then $\Omega$ is spatially-homogeneous and $1$-cover of $H(f)$.
\end{theorem}

\begin{theorem}\label{norma-hyper}
  Assume that the $\omega$-limit set $\Omega$ satisfies {\bf(H1)}. Then we have
\begin{itemize}
 \item[{ \rm (i)}] If $\dim V^u(\Omega)>0$, then there is a spatially-inhomogeneous minimal set $M\subset \O$ such that $\O\subset \Sigma M$.
Moreover, for any $g\in H(f)$, there exists $u_g\in X$ such that $p^{-1}(g)\cap \Omega\subset (\Sigma u_g,g)$;
 and there is a $C^1$-function $c^g:\mathbb{R}\to \mathcal{S}^1;t\mapsto c^g(t)$ (with its derivative $\dot{c}^g(t)$ being time-recurrent) such that
\begin{equation}\label{E:traveling-func}
\varphi(t,x,u_g,g)=u_{g\cdot t}(x+c^g(t)),
\end{equation}
where $\mathcal{S}^1=\mathbb{R}/L\mathbb{Z}$ and $L$ is the smallest common spatial-period of any element in $M$.

In particular, if $f$ in \eqref{equation-1} is uniformly almost-periodic in $t$, then the derivative  $\dot{c}^g(t)$ is almost-periodic in $t$.

 \item[{ \rm (ii)}] If $\dim V^u(\Omega)=0$, then $\Omega$ is spatially-homogeneous. Moreover, $\Omega$ contains at most two minimal sets and each minimal set is an almost $1$-cover of $H(f)$.
\end{itemize}
\end{theorem}

\begin{remark}
  {\rm Theorems \ref{hyperbolic0}-\ref{norma-hyper} indicate that, {\it when $\dim V^c(\Omega)\le 1$}, $\Omega$ is either spatially-homogeneous or spatially-inhomogeneous; and moreover, {\it any spatially-inhomogeneous $\O$ can be embedded into an $H(f)$-time-recurrent forced circle flow on $S^1$.} In particular, {\it $\O$ can be embedded into an almost-periodically forced cicle flow on $S^1$ if $f$ in \eqref{equation-1} is uniformly almost-periodic in $t$.} On the other hand, some example will be presented in the Appendix to indicate that such imbedding property can not hold anymore when $\dim V^c(\Omega)>1$. Consequently, these phenomena yield that there are essential differences between time-periodic cases (see, e.g. \cite{SF1}) and time almost-periodic cases.}
\end{remark}

\begin{theorem}\label{structure-thm}
Assume that the $\omega$-limit set $\Omega$ satisfies one of the hypotheses {\bf (H0)-(H2)}.
Then one of the following alternatives must hold:
\begin{itemize}
\item[{ \rm (i)}] There is a minimal set $M\subset \Omega$ such that $\Omega\subset \Sigma M$;

\item[{ \rm (ii)}] There is a minimal set $M_1\subset \Omega$ such that $\Omega\subset \Sigma M_1\cup M_{11}$, where  $M_{11}\neq \emptyset$ and $M_{11}$ connects $\Sigma M_1$ in the sense that if $(u_{11},g)\in M_{11}$, then $\Sigma M_1\cap \omega(u_{11},g)\not =\emptyset$ and $\Sigma M_1\cap\alpha (u_{11},g)\not =\emptyset$.

\item[{ \rm (iii)}] There are two minimal sets $M_1,M_2\subset \Omega$ with $\Sigma M_1\cap \Sigma M_2=\emptyset$ such that
 $\Omega\subset \Sigma M_1\cup \Sigma M_2\cup M_{12}$, where $M_{12}\not =\emptyset$, and for any $(u_{12},g)\in M_{12}$, $\omega(u_{12},g)\cap (\Sigma M_1\cup\Sigma M_2)\not=\emptyset$ and  $\alpha(u_{12},g)\cap (\Sigma M_1\cup\Sigma M_2)\not=\emptyset$.
\end{itemize}
Furthermore, given any spatially-inhomogeneous  minimal set $M\subset \Omega$, there is a residual subset $H_0(f)\subset H(f)$ such that,
for any $g\in H_0(f)$, there exists $u_g\in X$ such that $p^{-1}(g)\cap M\subset (\Sigma u_g,g)$; and moreover, the $C^1$-function $c^g(\cdot)$ in Theorem \ref{norma-hyper} is well-defined for each $g\in H_0(f)$.

In particular, if $f$ in \eqref{equation-1} is uniformly almost-periodic in $t$, then the derivative  $\dot{c}^g(t)$ is almost-automorphic in $t$.
\end{theorem}

\begin{remark}
 {\rm Theorem \ref{structure-thm} gives a complete classification of all the possible structures of the $\omega$-limit set $\O$ under the assumption {\bf (H0)}, or {\bf (H1)}, or {\bf (H2)}. Note that assuming {\bf (H0)} (resp. {\bf (H1)}), Theorem \ref{hyperbolic0}
(resp. Theorem \ref{norma-hyper}) in fact implies  Theorem \ref{structure-thm}. But we will give a direct proof of  Theorem \ref{structure-thm}. By Remark A.1(i) in the appendix and Theorem \ref{structure-thm}, the structure of the $\omega$-limit set $\O$  under the assumption {\bf (H2)} can be more complicated; and moreover, residually imbedding and almost automorphically forced circle flow may occur.}
  \end{remark}

 \begin{remark}
{\rm
The above three main Theorems are generalizations from autonomous and time-periodic cases (\cite{Massatt1986,Matano,SF1}) to general systems with time-recurrent structure which includes almost periodicity and almost automorphy. It
also deserves to point out that an almost periodically (automorphically) forced circle
flow has interesting and fruitful dynamical behavior (see, e.g. \cite{HuYi,Yi} and the references therein). The new phenomena we
discovered here reinforce the appearance of the almost periodically (automorphically) forced circle
flow on the $\omega$-limit set $\O$ of
the infinite-dimensional dynamical systems generated by evolutionary equations.
}
\end{remark}

In the forthcoming three Subsections \ref{proof-Thm5.1}-\ref{Proof-Thm5.3}, we will first prove Theorems \ref{structure-thm} in Subsection-\ref{proof-Thm5.1}. Based on this, we will then prove Theorem \ref{norma-hyper} in Subsection-\ref{Proof-Thm5.2}. Finally, in Subsection-\ref{Proof-Thm5.3}, we will prove Theorem \ref{hyperbolic0}.

\subsection{Proof of Theorem \ref{structure-thm}}\label{proof-Thm5.1}
In this subsection, we will prove Theorem \ref{structure-thm}. For this purpose, we first present a lemma on the structure of $\omega$-limit sets of the skew-product semiflow $\tilde \Pi^{t}$ on the induced quotient space in Section \ref{S:SP-S-quotient}.

\begin{lemma}\label{gen-hyper2}
Assume that the $\omega$-limit set $\Omega$ satisfies one of
the hypotheses {\bf (H0)}-{\bf (H2)}. Let $\tilde\Omega$ be defined in \eqref{E:Induced-Omega}. Then $\tilde\Omega$ contains at most two minimal sets of $\tilde \Pi^{t}$; and moreover, one of the following  three alternatives must occur:
\begin{itemize}
\item[{ \rm (i)}] $\tilde \O$ is a minimal invariant set of $\tilde \Pi^{t}$;
\item[{ \rm (ii)}] $\tilde \O=\tilde M_1\cup \tilde M_{11}$, where $\tilde M_1$ is minimal, $\tilde M_{11}\neq \emptyset$, $\tilde M_{11}$ connects $\tilde M_1$ in the sense that if $([u_{11}],g)\in \tilde M_{11}$, then $\omega([u_{11}],g)\cap \tilde M_1\neq \emptyset$, and $\alpha ([u_{11}],g)\cap \tilde M_1\neq \emptyset$;
\item[{ \rm (iii)}] $\tilde \O=\tilde M_1\cup \tilde M_2\cup \tilde M_{12}$, where $\tilde M_1$, $\tilde M_2$ are minimal sets, $\tilde M_{12}\neq \emptyset$ and connects $\tilde M_1$, $\tilde M_2$ in the sense that if $([u_{12}],g)\in \tilde M_{12}$, then $\omega([u_{12}],g)\cap(\tilde M_1\cup \tilde M_2)\neq\emptyset$ and $\alpha([u_{12}],g)\cap(\tilde M_1\cup \tilde M_2)\neq\emptyset$.
\end{itemize}
 \end{lemma}
\begin{proof}
Suppose that $\tilde\O$ contains three minimal sets $\tilde M_i$($i=1,2,3$) of $\tilde \Pi^{t}$. Then, by Corollary \ref{induced-property3}, one can find three  minimal sets $M_i\subset \O$($i=1,2,3$) such that $\tilde M_i=\{([u],g)|(u,g)\in M_i\}$ for $i=1,2,3$, respectively.

For each $g\in H(f)$ and $i=1,2,3$,  we define $m_i(g)$ and $M_i(g)$ as in \eqref{E:min-max}.
By virtue of Lemma \ref{induced1-1cover}(ii-b), we may assume without loss of generality that there is a $\delta>0$ such that
\begin{equation}\label{inequality}
  M_1(g)+\delta\leq m_2(g)\le M_2(g)\le m_3(g)-\delta,
\end{equation}
for all $g\in H(f)$.

Choose $(u_i,g_0)\in M_i\cap p^{-1}(g_0)\subset \O$, $i=1,2,3$. Recall that $\O:=\omega(u_0,g_0)$. Then there exists a sequence $t_n\to \infty$ such that $\Pi^{t_n}(u_0,g_0)\to(u_1,g_0)\in M_1$. Due to the compactness of $M_2$, one may also assume that
$\Pi^{t_n}(u_2,g_0)\to (\tilde u_2,g_0)$ for some $(\tilde u_2,g_0)\in M_2$. So,
Lemma \ref{constant-ontwoset} implies that there is $N_0\in \mathbb{N}$ such that $z(u_1-\sigma_a \tilde u_2)=N_0$ for all $a\in S^1$. Thus, by Corollary \ref{difference-lapnumber}(c) and compactness of $S^1$, there is a $T>0$ such that $z(\varphi(t,\cdot;u_0,g_0)-\varphi(t,\cdot;\sigma_au_2,g_0))\equiv N_0$, for all $a\in S^1$ and $t\ge T$. By Corollary \ref{difference-lapnumber}(b) and \eqref{inequality}, we obtain that $m(\varphi(t,\cdot;u_0,g_0))<m(\varphi(t,\cdot;u_2,g_0))$ for all $t\ge T$. Since $M_3\subset \omega(u_0,g_0)$, there exist some sequence  $t_n^\prime\to \infty$ and $g^*\in H(f)$ such that $m(\varphi(t_n^\prime,\cdot;u_0,g_0))\to m_3(g^*)$ as $n\to \infty$. Without loss of generality we may also assume that $m(\varphi(t_n^\prime,\cdot;u_2,g_0))\to \beta(g^*)$ with $\beta(g^*)\in [m_2(g^*),M_2(g^*)]$. As a consequence,
\[
m_3(g^*)\leq \beta(g^*)\le M_2(g^*),
\]
contradicting \eqref{inequality}. Thus, $\tilde\O$ contains at most two minimal sets.

Now, we can write $\tilde\O=\tilde M_1\cup\tilde M_2\cup\tilde M_{12}$, where $\tilde M_1$, $\tilde M_2$ are minimal sets of $\tilde \Pi^{t}$. When $\tilde M_1\neq \tilde M_2$, since $\tilde\O$ is connected, $\tilde M_{12}\neq \emptyset$. Choose some $([u_{12}],g)\in\tilde M_{12}$, then $\omega([u_{12}],g)\cap (\tilde M_1\cup \tilde M_2)$ and $\alpha([u_{12}],g)\cap (\tilde M_1\cup \tilde M_2)$ are nonempty. For otherwise, either $\omega([u_{12}],g)$ or $\alpha([u_{12}],g)$ will contain a new minimal set of $\tilde \Pi^{t}$; and hence, $\tilde\O$ will possess three minimal sets of $\tilde \Pi^{t}$, a contradiction. Thus, (iii) holds.  When $\tilde M_1=\tilde M_2$ (i.e., $\tilde \O$ contains a unique minimal set), then $\tilde M_{12}=\emptyset$ will imply (i); and if $\tilde M_{12}\neq \emptyset$, then a similar argument shows that $\omega([u_{12}],g)\cap\tilde M_1\neq \emptyset$, $\alpha([u_{12}],g)\cap \tilde M_1\neq \emptyset$ for any $([u_{12}],g)\in \tilde M_{12}$. The proof of this lemma is completed.
\end{proof}

We are ready to prove Theorem \ref{structure-thm}.

\begin{proof}[Proof of Theorem \ref{structure-thm}]
Recall that $\tilde \O=\{([u],g)|(u,g)\in \Omega\}$. When Lemma \ref{gen-hyper2}(i) holds, one has $\tilde \O=\tilde M$; and hence, Corollary \ref{induced-property3} implies that there is a minimal set $M\subset \O$ such that $\tilde M=\{([u],g)|(u,g)\in M\}$. Suppose that there is $(u_*,g)\in \O$, but $(u_*,g)\notin \Sigma M$. Then $u_*\neq \sigma_a u$ for any $a\in S^1$ and $(u,g)\in M$, which means that $([u_*],g)\notin \tilde M$, a contradiction to $([u_*],g)\in \tilde \O=\tilde M$. Thus, $\O\subset \Sigma M$.

When Lemma \ref{gen-hyper2}(ii) holds, that is, $\tilde \O=\tilde M_1\cup \tilde M_{11}$, where $\tilde M_1$ is a minimal set of $\tilde \Pi^{t}$, $\tilde M_{11}\neq \emptyset$. By Corollary \ref{induced-property3} again, one can choose a minimal set $M_1\subset \O$ such that $\tilde M_1=\{([u],g)|(u,g)\in M_1\}$. Let $M_{11}=\O\setminus \Sigma M_1$. Then it is easy to see that $\tilde M_{11}=\{([u],g)|(u,g)\in M_{11}\}$; and moreover, there is no minimal set in $M_{11}$.
So, we can assert that both $\Sigma M_1\cap \omega(u_{11},g)\ne\emptyset$ and $\Sigma M_1\cap \alpha(u_{11},g)\ne\emptyset$. In fact, suppose for instance that $\Sigma M_1\cap \omega(u_{11},g)=\emptyset$. Then one can find a minimal set $M_2\subset \omega(u_{11},g)$. So, $M_2\cap \Sigma M_1=\emptyset$; and hence, $\Sigma M_2\cap \Sigma M_1=\emptyset$. Let $\tilde M_2=\{([u],g)|(u,g)\in M_2\}$, then $\tilde M_2\neq \tilde M_1$ is also a minimal set of $\tilde\Pi^t$ contained in $\tilde \Omega$, a contradiction. Thus, we have proved (ii). Similarly, we can also prove (iii) as long as Lemma \ref{gen-hyper2}(iii) holds.

Now let $M\subset \Omega$ be any spatially-inhomogeneous minimal set. Since { one of } {\bf(H0)}-{\bf(H2)} holds,
Remark \ref{invari-space} entails that $\dim V^c(M)\le 2$. Since $M$ is spatially-inhomogeneous, Lemma \ref{hyperbolic-minimal}(1) implies that $\dim V^c(M)>0$; and moreover, Lemma \ref{hyperbolic-minimal}(2) further implies that if $\dim V^c(M)=1$ then we must have $\dim V^u(M)>0$. Thus, we have obtained that either $\dim V^c(M)=1$ with $\dim V^u(M)>0$,  or $\dim V^c(M)=2$ with $\dim V^u(M)$ being odd. As a consequence, it follows from Lemma \ref{normhyper-minimal}(i)-(ii) that there exists at least a residual subset $H_0(f)\subset H(f)$ such that
for any $g\in H_0(f)$, there exists $u_g\in X$ such that $M\cap p^{-1}(g)\subset (\Sigma u_g,g)$.

Finally, we will show the existence of $c^g(t)$ which satisfies \eqref{E:traveling-func}.
The following argument is essentially adapted from \cite{SWZ}. For completeness we give more detail here. By Lemma \ref{induced1-1cover}(i), we obtain the induced minimal set $\tilde{M}$, which is an almost $1$-cover of $H(f)$. Define the mapping
\begin{equation}\label{E:natu-proj-h}
h:\tilde{M}\to \mathbb{R}^1\times H(f); ([u],g)\mapsto (m(u),g).
\end{equation}
Let $\hat{M}=h(\tilde{M})$. Clearly, $h$ is well-defined and continuous onto $\hat{M}$. Moreover, $h$ is injective due to Lemma \ref{normhyper-minimal}(iv). Recall that $\tilde{M}$ and $\hat{M}$ are both compact, $h$ is also a closed mapping. Hence $h$ is a homeomorphism from $\tilde{M}$ onto $\hat{M}$. On such $\hat{M}\subset \mathbb{R}^1\times H(f)$, one can naturally define the skew-product flow
\begin{equation}\label{E:induced-flow-hat-M}
\hat{\Pi}^t:\hat{M}\to \hat{M};(m(u),g)\mapsto (m(\varphi(t,\cdot,u,g)),g\cdot t),
\end{equation}
which is induced by $\Pi^t$ restricted to $M$. So, a straightforward check yields that
$$h\circ\tilde{\Pi}^t([u],g)=\hat{\Pi}^t\circ h([u],g)\,\,\,
\text{ for any }([u],g)\in \tilde{M}.$$ This entails that $h$ is a topologically-conjugate homeomorphism between $\tilde{M}\to \hat{M}\subset \mathbb{R}^1\times H(f)$. Hence, $\hat{M}$ is an almost $1$-cover, since $\tilde{M}$ is an almost $1$-cover (with the residual subset $H_0(f)\subset H(f)$).

For each $g\in H_0(f)$, we choose some element, still denoted by $u_g(\cdot)$, from the $S^1$-group orbit $\Sigma u_g$ such that
\begin{equation}\label{E:u-g--base-trn1}
u_g(0)=m(u_g),\,\,\,\,\,\,\,\hat{M}\cap p^{-1}(g)=(m(u_g),g)\,\, \text{and }\tilde{M}\cap \tilde{p}^{-1}(g)=([u_g],g).
\end{equation}
Then it follows from \eqref{E:induced-flow-hat-M} and \eqref{E:u-g--base-trn1} that $$u_{g\cdot t}(0)=m(u_{g\cdot t})=m(\varphi(t,\cdot,u_g,g))\,\,\,\text{ for any }g\in H_0(f)\,\,\text{ and }t\in \mathbb{R}.$$
As a consequence, for each $g\in H_0(f)$, the function $t\mapsto u_{g\cdot t}(0)$  is clearly continuous and time-recurrent in $t$ (almost automorphic in $t$, if $f$ is uniformly almost periodic in $t$) due to the fact that $\hat{M}$ is an almost $1$-cover; and moreover, $u_{g\cdot t}(x)$ is time-recurrent (almost automorphic) in $t$ uniformly in $x$.

%
%

Due to the spatial-inhomogeneity of $M$,  it follows that $\varphi_x(t,\cdot;u_g,g)\in V^c(\Pi^t(u_g,g))$ for any $t$. Recall that $M$ satisfies
  either $\dim V^c(M)=1$ with $\dim V^u(M)>0$,  or $\dim V^c(M)=2$ with $\dim V^u(M)$ being odd.
  Then Lemma \ref{L:zero-inva} implies that $\varphi_x(t,\cdot,u_g,g)$ only has simple zeros for any $t\in \mathbb{R}$. In particular, by letting $t=0$, one has $u_g^{'}(\cdot)$ only has simple zeros. Together with $u_g^\prime(0)=0$ (because $u_g(0)=m(u_g)$), this then implies that
  \begin{equation}\label{E:circle-flow-eq11}
u^{''}_g(0)\not =0\quad \text{ for any }g\in H_0(f).
\end{equation}

Now, define a nonnegative function $t\mapsto c^g(t)\ge 0$ (with $g\in H_0(f)$) such that
\begin{equation}\label{E:rotation-spiral1}
 \varphi(t,x;u_g,g)= u_{g\cdot t}(x+ c^g(t)),\quad \text{or equivalently,}\quad \varphi(t,x-c^g(t);u_g,g)=u_{g\cdot t}(x).
\end{equation}
Let $L\in (0,2\pi]$ be the smallest common spatial-period of the elements in the minimal set $M$ and $\mathcal{S}^1:=\mathbb{R}/L\mathbb{Z}$. Then for each $t\in \mathbb{R}$, one can further choose $c^g(t)\in \mathcal{S}^1$ so that $c^g(t)$ is continuous in $t$. Indeed, suppose that there is a sequence $t_n\to t_0$ such that $\abs{c^g(t_0)-c^g(t_n)}\ge \epsilon_0>0$ in $\mathcal{S}^1$. For the sake of simplicity, we assume $c^g(t_n)\to c^*$ with $c^*\in \mathcal{S}^1$. So, $c^g(t_0)\ne c^*$ in $\mathcal{S}^1$. On the other hand, by \eqref{E:rotation-spiral1}, one has
$u_{g\cdot t_0}(x+ c^g(t_0))=\varphi(t_0,x;u_g,g)=\lim_{n\to\infty}\varphi(t_n,x;u_g,g)=\lim_{n\to\infty}u_{g\cdot t_n}(x+ c^g(t_n))=u_{g\cdot t_0}(x+ c^*)$. This contradicts $c^g(t_0)\ne c^*$ with $c^g(t_0),c^*\in \mathcal{S}^1$, because $L$ is the minimal spatial-period.
So, the function $t\mapsto c^g(t)\in \mathcal{S}^1$ is continuous.

By \eqref{E:rotation-spiral1} and the property of $u_{g\cdot t}(x)$ in \eqref{E:circle-flow-eq11}, we observe that
$$
\varphi_x(t,-c^g(t);u_g,g)=u^{'}_{g\cdot t}(0)=0\,\,
\text{ and }\,\,\varphi_{xx}(t,-c^g(t);u_g,g)= u^{''}_{g\cdot t}(0)\not =0.
$$
Then by the continuity of $c^g(t)$ in $t$ and Implicit Function Theorem, we have $c^g(t)$ is differentiable in $t$; and moreover, we have
\begin{eqnarray}\label{circle-flow-eq41}
\dot {c}^g(t)=G(t,c^g(t)),
 \end{eqnarray}
 where
\begin{equation*}
\begin{split}
G(t,z)=&g_p(t,\varphi(t,-z;u_g,g),\varphi_x(t,-z;u_g,g))\\
&+\dfrac{\varphi_{xxx}(t,-z;u_g,g)+g_u(t,\varphi(t,-z;u_g,g),\varphi_x(t,-z;u_g,g)\varphi_x(t,-z;u_g,g)}{\varphi_{xx}(t,-z;u_g,g)},\,\,\,\text{ for }z\in \mathcal{S}^1.
\end{split}
\end{equation*}
It is easy to see that $G(t,z+L)=G(t,z)$ and the function $G(t,c^g(t))=g_p(t,u_{g\cdot t}(0),0)+\frac{u^{'''}_{g\cdot t}(0)}{u^{''}_{g\cdot t}(0))}$, and hence $\dot{c}^g(t)$, is time-recurrent (resp. time almost-automorphic in $t$ if $f$ is uniformly almost periodic in $t$). Thus, we have obtained that \eqref{E:rotation-spiral1} and \eqref{circle-flow-eq41}, which naturally induces
a time-recurrently (resp. almost automorphically) forced skew-product flow on $\mathcal{S}^1\times H(f)$.
The proof of this theorem is completed.
\end{proof}

\subsection{Proof of Theorem \ref{norma-hyper}}\label{Proof-Thm5.2}

In this subsection, we will prove Theorem \ref{norma-hyper}. Since the proof of Theorem \ref{norma-hyper}(ii) is similar to \cite[Theorem 5.1 (ii)]{SWZ2}, in the rest of this section we only prove Theorem \ref{norma-hyper}(i).

\begin{proof}[Proof of Theorem {\rm \ref{norma-hyper}(i)}]
Since {\bf (H1)} holds and $\dim V^u(\Omega)>0$, Lemma \ref{induced1-1cover}(i) implies that any minimal set $\tilde M$ of $\tilde{\Pi}^t$ is a $1$-cover of $H(f)$. In the following, we will show that $\tilde{\Omega}=\tilde M$ for some minimal set $\tilde M$ of $\tilde{\Pi}^t$; that is, there is a minimal set $M\subset \O$ of $\Pi^t$ such that $\O\subset \Sigma M$.
 To this end,
it suffices to show that cases (ii)-(iii) in Lemma \ref{gen-hyper2} can not occur.

Suppose that case (ii) in Lemma \ref{gen-hyper2} occurs. Then $\tilde \O=\tilde M_1\cup \tilde M_{11}$ where $\tilde M_1$ is minimal and $\tilde M_{11}\neq \emptyset$. So, Lemma \ref{gen-hyper2}(ii) implies that $\{([u_1],g),([u_{11}],g)\}$ is a two sided proximal pair for any $([u_1],g)\in \tilde M_1$ and $([u_2],g)\in \tilde M_{11}$. This contradicts to Lemma \ref{inducedproxi}. So, the case (ii) in Lemma \ref{gen-hyper2} can not happen.

Suppose that case (iii) in Lemma \ref{gen-hyper2} occurs.
Then $\tilde\O=\tilde M_1\cup\tilde M_2\cup\tilde M_{12}$, where $\tilde M_1$ and $\tilde M_2$ are minimal sets and $\tilde M_{12}\neq \emptyset$. By Corollary \ref{induced-property3}, there are two minimal sets $M_i\subset\O$($i=1,2$) with $\Sigma M_1\cap \Sigma M_2=\emptyset$, such that $\tilde M_i=\{([u_i],g)|(u_i,g)\in M_i\}$ for $i=1,2$. Since $\tilde M_i$($i=1,2$) are $1$-cover of $H(f)$, we may assume without loss of generality that
\begin{equation}\label{E:back-for-appro-11}
\tilde \Pi^t([u_1],g)-\tilde \Pi^t([u_{12}],g)\to 0\, \text{ and }\tilde \Pi^{-t}([u_2],g)-\tilde \Pi^{-t}([u_{12}],g)\to 0,\, \text{ as }t\to \infty,
\end{equation}
for any $([u_i],g)\in\tilde M_i$ ($i=1,2$) and $([u_{12}],g)\in\tilde M_{12}$. Moreover, $\Omega\subset\Sigma M_1\cup \Sigma M_2\cup M_{12}$, where $M_{12}\subset\Omega$ ($M_{12}$ contains no minimal set of $\Omega$) such that $\tilde{M}_{12}=\{([u],g)|(u,g)\in M_{12}\}$. We will discuss the following three alternatives separately:

(i) Both $M_1$ and $M_2$ are spatially-homogeneous;

(ii) Both $M_1$ and $M_2$ are spatially-inhomogeneous;

(iii) One is spatially-homogeneous, the other is spatially-inhomogeneous.

\noindent For each case, we will deduce certain contradiction (see the forthcoming three sub-lemmas) . This then makes that the case (iii) in Lemma \ref{gen-hyper2} can not happen.
\vskip 3mm

\noindent {\it Sub-Lemma 1: Alternative {\bf (i)} cannot occur.}
\begin{proof}\vskip -2mm
Suppose that both $M_1$ and $M_2$ are spatially-homogeneous. Then Lemma \ref{homogeneous} implies that $M_i$ ($i=1,2$) are hyperbolic and satisfying Lemma \ref{homogeneous}(b).

Let $(u_{12},g)\in M_{12}$ and $(u_i,g)\in p^{-1}(g)\cap M_i$ ($i=1,2$). Then Lemma \ref{proximalcons} entails that neither  $\{(u_1,g),(u_{12},g)\}$ nor $\{(u_1,g),(u_{12},g)\}$ is a two sided proximal pair. Thus, together with \eqref{E:back-for-appro-11}, we can assume

\begin{equation}\label{positive-ays}
 \Pi^{t}(u_{12},g)-\Pi^{t}(u_{1},g)\to 0 \ \text{as} \ t\to\infty
\end{equation}
and
\begin{equation*}\label{negative-ays}
 \Pi^{t}(u_{12},g)-\Pi^{t}(u_{2},g)\to 0 \ \text{as} \ t\to -\infty.
\end{equation*}

Since both $M_1$ and $M_2$ are spatially-homogenous, $z(\varphi(t,\cdot;u_1,g)-\varphi(t,\cdot;u_2,g))=0$ for all $t\in \mathbb{R}$. Let $t_n\to -\infty$ be such that $\Pi^{t_n}(u_{12},g)\to (u_2,g)$ and $\Pi^{t_n}(u_{1},g)\to(u_1,g)$, as $n\to\infty$. Then, by Lemma \ref{zero-cons-local}, there is $N\in\mathbb{N}$ such that,
 \begin{equation}\label{equa-limit}
   z(u_1-u_2)=z(\varphi(t_n,\cdot;u_{12},g)-\varphi(t_n,\cdot;u_1,g))=0
 \end{equation}
for any $n>N$.

We will consider the two cases that $\dim V^u(\Omega)$ is odd and $\dim V^u(\Omega)$ is even separately. When $\dim V^u(\Omega)$ is odd, by virtue of Lemma \ref{homogeneous}(b), one can choose $\delta^*>0$ so small that
\begin{equation*}\label{E:cs=unstable}
M^{u}(\omega,\delta^*)=\tilde M^u(\omega,\delta^*)\,\text{ and }\,M^{cs}(\omega,\delta^*)=\tilde M^{s}(\omega,\delta^*),
\end{equation*}
where $\tilde M^u,\tilde M^s$ denote respectively the local unstable and stable manifolds of $\omega\in M$ with respect to the Sacker-Sell spectrum $\sigma(M)$ (see more discussion in \cite[(5.10) and Remark 5.1]{SWZ2}). So, by virtue of \eqref{positive-ays} and Remark \ref{stable-leaf}(1), we have $\varphi(t,\cdot;u_{12},g)\in \tilde M^{s}(\Pi^t(u_{1},g),\delta^*)=M^{cs}(\Pi^t(u_{1},g),\delta^*)$ for $t\gg 1$. Recall that $N_u=\dim V^u(\O)+1$. Then Lemma \ref{zerocenter}(2) entails that
 \begin{equation*}
   z(\varphi(t,\cdot;u_{12},g)-\varphi(t,\cdot;u_1,g))\geq N_u\quad t\gg 1;
 \end{equation*}
and hence, { Corollary \ref{difference-lapnumber}(a) implies that}
\begin{equation*}
   z(\varphi(t,\cdot;u_{12},g)-\varphi(t,\cdot;u_1,g))\geq N_u \quad t\in \mathbb{R}^1,
 \end{equation*}
a contradiction to \eqref{equa-limit}.

When $\dim V^u(\Omega)$ is even, also by Lemma \ref{homogeneous}(b), one can choose $\delta^*>0$ so small that
\begin{equation}\label{E:cu=unstable}
M^{cu}(\omega,\delta^*)=\tilde M^u(\omega,\delta^*)\,\text{ and }\,M^{s}(\omega,\delta^*)=\tilde M^{s}(\omega,\delta^*).
\end{equation}
(see more discussion in \cite[(5.12) and Remark 5.1]{SWZ2}). Thus, by virtue of \eqref{positive-ays} and Remark \ref{stable-leaf}(1), we have $\varphi(t,\cdot;u_{12},g)\in \tilde M^{s}(\Pi^t(u_{1},g),\delta^*)=M^{s}(\Pi^t(u_{1},g),\delta^*)$ for $t\gg 1$. Note that $N_u=\dim V^u(\O)$. By Lemma \ref{zerocenter}(3), one has
\begin{equation*}
   z(\varphi(t,\cdot;u_{12},g)-\varphi(t,\cdot;u_1,g))\geq N_u+2\quad t\gg 1;
 \end{equation*}
and hence,  Corollary \ref{difference-lapnumber}(a) implies that
\begin{equation*}
   z(\varphi(t,\cdot;u_{12},g)-\varphi(t,\cdot;u_1,g))\geq N_u+2 \quad t\in \mathbb{R}^1,
 \end{equation*}
which is also a contradiction to \eqref{equa-limit}.

This implies that Alternative {\bf (i)} cannot occur.
\end{proof}

\noindent {\it Sub-Lemma 2: Alternative {\bf (ii)} cannot occur.}
\begin{proof}\vskip -2mm
By virtue of \eqref{E:back-for-appro-11}, one can find $ a^*_{12},a^{**}_{12}\in S^1$, $(u^*,g^*)\in M_1$, $(u^{**},g^{**})\in M_2$, and two subsequences $t_n\to\infty$,  $s_n\to-\infty$ such that
\begin{equation}\label{bnormalposi-proxi1}
\begin{split}
 \Pi^{t_n}(u_1,g)\to(u^*,g^*) \,\,\text{ and }\,\,
\Pi^{t_n}(\sigma_{a^{*}_{12}}u_{12},g)\to(u^*,g^*)
\end{split}
\end{equation}
and
\begin{equation}\label{bnormalnega-proxi1}
\begin{split}
 \Pi^{s_n}(u_2,g)\to(u^{**},g^{**}) \,\,\text{ and }\,\,
\Pi^{s_n}(\sigma_{a^{**}_{12}}u_{12},g)\to(u^{**},g^{**}).
\end{split}
\end{equation}
So, similarly as \eqref{geq-estimate} and \eqref{leq-estimate2}, one can obtain that
\begin{equation}\label{b3-large-estimate}
  z(\varphi(t,\cdot;u_1,g)-\varphi(t,\cdot;\sigma_au_{12},g))\geq N_u
\end{equation}
and
\begin{equation}\label{b3-small-estimate}
  z(\varphi(t,\cdot;u_2,g)-\varphi(t,\cdot;\sigma_bu_{12},g))\leq N_u
\end{equation}
for all $t\in \mathbb{R},a,b\in S^1$.
Moreover, by Lemma \ref{constant-ontwoset}, there exists some $N\in\mathbb{N}$ such that
\begin{equation*}
  z(\varphi(t,\cdot;\sigma_bu_1,g)-\varphi(t,\cdot;\sigma_au_2,g))= N
\end{equation*}
for all $t\in \mathbb{R},a,b\in S^1$. We will show that $N=N_u$. In fact, since $M_1$ is compact, there exist $\{s_{n_k}\}\subset \{s_n\}$ and $(u^{**}_1,g^{**})\in M_1$ such that $\Pi^{s_{n_k}}(u_1,g)\to (u^{**}_1,g^{**})$. Then it follows from \eqref{b3-large-estimate} that $z(u^{**}_1-u^{**})\geq N_u$. Similarly, by using \eqref{b3-small-estimate}, one has $z(u^{*}_2-u^{*})\leq N_u$, for some $(u^{*}_2,g^*)\in M_2$. Again by Lemma \ref{constant-ontwoset}, we have $z(u^{*}_2-u^{*})=z(u^{**}_1-u^{**})$. Therefore, $N=N_u$. Moreover, we have
\begin{equation}\label{constant-A}
  z(\varphi(t,\cdot;u_1,g)-\varphi(t,\cdot;\sigma_au_{12},g))= N_u
\end{equation}
and
\begin{equation}\label{constant-B}
  z(\varphi(t,\cdot;u_2,g)-\varphi(t,\cdot;\sigma_bu_{12},g))= N_u
\end{equation}
for all $t\in \mathbb{R}$ and $a,b\in S^1$.
\vskip 2mm

When $\dim V^u(\O)$ is odd (resp. $\dim V^u(\O)$ is even), it follows from \eqref{constant-B} (resp. \eqref{constant-A}), Lemma \ref{zero-cons-local} and the compactness of $S^1$ that there exists $\delta>0$ (independent of $a\in S^1$) such that
\begin{equation}\label{case-b2-const-1}
z(\sigma_a u_{12}-u_2+v)=N_u \quad(\mathrm{resp}.\ z(\sigma_a u_{12}-u_1+v)=N_u)
\end{equation}
for any $a\in S^1$ and $\|v\|<\delta$. According to \eqref{bnormalnega-proxi1} (resp. \eqref{bnormalposi-proxi1}) and Lemma \ref{hyperbolic1}, there exists some $v_n\in M^s(\varphi(s_n,\cdot;\sigma_{a^{**}_{12}}u_{12},g),g\cdot s_n,\delta^*)\cap M^{cu}(\varphi(s_n,\cdot;u_2,g),g\cdot s_n,\delta^*)$ (resp. $v_n\in M^u(\varphi(t_n,\cdot;\sigma_{a^{*}_{12}}u_{12},g),g\cdot t_n,\delta^*)\cap M^{cs}(\varphi(t_n,\cdot;u_1,g),g\cdot t_n,\delta^*)$) for $n$ sufficiently large. Similarly as the assertion between \eqref{case-A-const-1}-\eqref{backward-contracting-1}, we can also obtain that $v_n\notin M^{c}(\varphi(s_n,\cdot;u_2,g),g\cdot s_n,\delta^*)$ (resp. $v_n\notin M^{c}(\varphi(t_n,\cdot;u_1,g),g\cdot t_n,\delta^*)$). Recall that Lemma \ref{homogeneous}(a) implies $M^c(u_2,g,\delta^*)\subset\Sigma u_2$
(resp. $M^c(u_1,g,\delta^*)\subset \Sigma u_1$) for $\delta^*$ sufficiently small, the foliation statement in Remark \ref{stable-leaf}(4) entails that there is $\tilde a_n\in S^1$ such that $v_n\in M^u(\varphi(s_n,\cdot;\sigma_{\tilde a_n}u_2,g),g\cdot s_n,\delta^*)$ (resp. $v_n\in M^s(\varphi(t_n,\cdot;\sigma_{\tilde a_n}u_1,g),g\cdot t_n,\delta^*)$ ). So, by Lemma \ref{zerocenter}(2) (resp. Lemma \ref{zerocenter}(3)), we have
\begin{equation}\label{case-b2-less-1}
\begin{split}
& z(\varphi(-s_n,\cdot;v_n,g\cdot s_n)- \sigma_{\tilde a_n}u_2)\leq N_u-2\\
&(\mathrm{resp}.\ z(\varphi(-t_n,\cdot;v_n,g\cdot t_n)- \sigma_{\tilde a_n}u_1)\geq N_u+2).
\end{split}
\end{equation}
On the other hand, since $v_n\in M^s(\varphi(s_n,\cdot;\sigma_{a^{**}_{12}}u_{12},g),g\cdot s_n,\delta^*)$ (resp. $v_n\in M^u(\varphi(t_n,\cdot;\sigma_{a^{*}_{12}}u_{12},g),g\cdot t_n,\delta^*)$), one has
\begin{equation}\label{E:backward-contract-11}
\begin{split}
&\norm{\varphi(-s_n,\cdot;\sigma_{-\tilde a_n}v_n,g\cdot s_n)-\sigma_{a^{**}_{12}-\tilde a_n}u_{12}}\\
=&\norm{\varphi(-s_n,\cdot;v_n,g\cdot s_n)-\sigma_{a^{**}_{12}}u_{12}}\leq Ce^{\frac{\alpha}{2} s_n}\|v_n-\varphi(s_n,\cdot;\sigma_{a^{**}_{12}}u_{12},g)\|\le C\delta^*e^{\frac{\alpha}{2} s_n}\\
(\mathrm{resp}.\,\,\,\, &\norm{\varphi(-t_n,\cdot;\sigma_{-\tilde a_n}v_n,g\cdot t_n)-\sigma_{a^{*}_{12}-\tilde a_n}u_{12}}\\
=&\norm{\varphi(-t_n,\cdot;v_n,g\cdot t_n)-\sigma_{a^{*}_{12}}u_{12}}
\leq Ce^{-\frac{\alpha}{2} t_n}\|v_n-\varphi(t_n,\cdot;\sigma_{a^{*}_{12}}u_{12},g)\|\le C\delta^*e^{-\frac{\alpha}{2} t_n}).
\end{split}
\end{equation}
Together with \eqref{case-b2-const-1}, \eqref{E:backward-contract-11} immediately implies that
\begin{equation*}
\begin{split}
  &z(\varphi(-s_n,\cdot;v_n,g\cdot s_n)- \sigma_{\tilde a_n}u_2)=z(\varphi(-s_n,\cdot;\sigma_{-\tilde a_n}v_n,g\cdot s_n)-u_2)\\
  =&\, z(\varphi(-s_n,\cdot;\sigma_{-\tilde a_n}v_n,g\cdot s_n)-\sigma_{a^{**}_{12}-\tilde a_n}u_{12}+\sigma_{a^{**}_{12}-\tilde a_n}u_{12}-u_2)=z(\sigma_{a^{**}_{12}-\tilde a_n}u_{12}-u_2)=N_u\\
  (\mathrm{resp}.\,\,\,\, & z(\varphi(-t_n,\cdot;v_n,g\cdot t_n)- \sigma_{\tilde a_n}u_1)=z(\varphi(-t_n,\cdot;\sigma_{-\tilde a_n}v_n,g\cdot t_n)-u_1)\\
  =&\, z(\varphi(-t_n,\cdot;\sigma_{-\tilde a_n}v_n,g\cdot t_n)-\sigma_{a^{*}_{12}-\tilde a_n}u_{12}+\sigma_{a^{*}_{12}-\tilde a_n}u_{12}-u_1)=z(\sigma_{a^{*}_{12}-\tilde a_n}u_{12}-u_1)=N_u),
\end{split}
\end{equation*}
 for $n$ sufficiently large. Thus, we have obtained a contradiction to \eqref{case-b2-less-1}, which implies that Alternative {\bf (ii)} cannot occur.
\end{proof}

\noindent {\it Sub-Lemma 3: Alternative {\bf (iii)} cannot occur.}
\begin{proof}\vskip -2mm
Without loss of generality, we assume that $M_1$ is spatially-homogeneous and $M_2$ is spatially-inhomogeneous.  For any $\omega\in M_1$, we still denote by $\tilde M^u(\omega,\delta^*)$ and $\tilde M^s(\omega,\delta^*)$ the local unstable and stable manifolds of $\omega$ with respect to the Sacker-Sell spectrum $\sigma(M_1)$. Since $M_1$ is spatially-homogeneous, the first statement in \eqref{E:back-for-appro-11} implies that $\Pi^t(u_1,g)-\Pi^t(\sigma_{a} u_{12},g)\to 0$ as $t\to \infty$ for any $a\in S^1$. Due to Remark \ref{stable-leaf}(1), this means that
\begin{equation}\label{E:forward-homo-back-inhomo-1}
{ \varphi(t,\cdot;\sigma_{a}u_{12},g)\in \tilde M^{s}((\Pi^t(u_{1},g),\delta^*)\,\, \text{ for }t\gg 1.}
 \end{equation}
While in the backward time-direction,  by Lemma \ref{proximal-equivalent}, we have $\tilde \Pi^{-t}([u_2],g)-\tilde \Pi^{-t}([u_{12}],g)\to 0$ ($t\to +\infty$) implies that, there are  $a^*_{12}\in S^1$, $(u_2^*,g^*)\in M_2$ and a sequence $s_n\to-\infty$, such that
\begin{equation}\label{bnormalnega-proxi2}
\begin{split}
 \Pi^{s_n}(u_2,g)\to(u^{*}_2,g^{*}) \,\,\text{ and }\,\,
 \Pi^{s_n}(\sigma_{a^{*}_{12}}u_{12},g)\to(u^{*}_2,g^{*}).
\end{split}
\end{equation}
In the following, we will again consider two cases, that is, $\dim V^u(\O)$ is even or odd, separately.

{\it Case (A): $\dim V^u(\O)$ is even}. According to \cite[Remark 5.1(ii)]{SWZ2}, $\tilde M^u(\omega,\delta^*)=M^{cu}(\omega,\delta^*)$ and $\tilde M^s(\omega,\delta^*)=M^s(\omega,\delta^*)$ for any $\omega\in M_1$. So, \eqref{E:forward-homo-back-inhomo-1} implies that $\varphi(t,\cdot;\sigma_{a}u_{12},g)\in M^{s}(\Pi^t(u_{1},g),\delta^*)$ for all $a\in S^1$ and $t\gg 1$. So, Lemma \ref{zerocenter}(3) implies that
\begin{equation*}
   z(\varphi(t,\cdot;\sigma_au_{12},g)-\varphi(t,\cdot;u_1,g))\geq N_u+2,
 \end{equation*}
for all $a\in S^1$ and $t\gg 1$. Together with Corollary \ref{difference-lapnumber}(a), it follows that
 \begin{equation}\label{Case-A-estimate-geq}
   z(\varphi(t,\cdot;\sigma_au_{12},g)-\varphi(t,\cdot;u_1,g))\geq N_u+2\quad
   \text{ for all }t\in \mathbb{R}^1 \text{ and }a\in S^1.
 \end{equation}
For simplicity, we assume that $\Pi^{s_n}(u_1,g)\to (u^*_1,g^*)\in M_1$.
 Combining with \eqref{bnormalnega-proxi2}-\eqref{Case-A-estimate-geq}, we have
 \begin{equation}
 z(u_2^*-u_1^*)=z(\varphi(s_n,\cdot;\sigma_{a^*_{12}}u_{12},g)-\varphi(s_n,\cdot;u_1,g))\geq N_u+2,
 \end{equation}
for $n\gg 1$. Note also that $(u_i^*,g^*)\in M_i (i=1,2)$, then $N\ge N_u+2$, where $N$ is as defined in Lemma \ref{constant-ontwoset}.

On the other hand, similarly as \eqref{leq-estimate2}, one can also use \eqref{bnormalnega-proxi2} and Lemma \ref{zerocenter}(3) to obtain that
 \begin{equation}\label{Case-A-estimate-leq}
   z(\varphi(t,\cdot;\sigma_au_{12},g)-\varphi(t,\cdot;u_2,g))\leq N_u, \quad \text{ for all }t\in \mathbb{R}^1 \text{ and }a\in S^1.
 \end{equation}
Since $\Pi^t(u_1,g)-\Pi^t(u_{12},g)\to 0$ as $t\to \infty$, one can choose a subsequence $t_n\to \infty$ such that $\Pi^{t_n}(u_{12},g)\to (u^{**}_1,g^{**})\in M_1$. For this sequence $t_n\to \infty$, we can also assume that  $\Pi^{t_n}(u_{2},g)\to (u^{**}_2,g^{**})\in M_2.$ Again, by Lemma \ref{constant-ontwoset} and \eqref{Case-A-estimate-leq} , we obtain that
\begin{equation}\label{b3-small-estimate2}
N=z(u^{**}_2-u^{**}_1)=z(\varphi(t_n,\cdot;u_{12},g)-\varphi(t_n,\cdot;u_2,g))\leq N_u,
\end{equation}
a contradiction to $N\ge N_u+2$.

\vskip 3mm
{\it Case (B): $\dim V^u(\O)$ is odd.} According to \cite[Remark 5.1(i)]{SWZ2}, $\tilde M^s(\omega,\delta^*)=M^{cs}(\omega,\delta^*)$ for any $\omega\in M_1$. So, \eqref{E:forward-homo-back-inhomo-1} implies that $\Pi^t(\sigma_{a}u_{12},g)\in M^{cs}(\Pi^t(u_{1},g),\delta^*)$ for all $a\in S^1$ and $t\gg 1$.
 Thus, Lemma \ref{zerocenter}(2) implies that
\begin{equation*}
   z(\varphi(t,\cdot;\sigma_au_{12},g)-\varphi(t,\cdot;u_1,g))\geq N_u,
 \end{equation*}
for all $a\in S^1$ and $t\gg 1$. Together with Corollary \ref{difference-lapnumber}(a), this implies that
\begin{equation}\label{Case-B-estimate-geq}
   z(\varphi(t,\cdot;\sigma_au_{12},g)-\varphi(t,\cdot;u_1,g))\geq N_u, \quad\forall t\in \mathbb{R}^1,\ a\in S^1.
\end{equation}
Similarly as \eqref{leq-estimate}, it follows from \eqref{bnormalnega-proxi2}, that
\begin{equation}\label{Case-B-estimate-leq}
   z(\varphi(t,\cdot;\sigma_au_{12},g)-\varphi(t,\cdot;u_2,g))\leq N_u, \quad\forall t\in \mathbb{R}^1,\ a\in S^1.
 \end{equation}
Thus, by repeating the arguments between \eqref{Case-A-estimate-geq}-\eqref{b3-small-estimate2}, one can obtain that
 $$ z(\varphi(t,\cdot;u_{1},g)-\varphi(t,\cdot;\sigma_bu_2,g))= N_u$$
 for all $b\in S^1$ and $t\in\mathbb{R}$. As a consequence, it is also not difficult to see that
 $$ z(\varphi(t,\cdot;u_{2},g)-\varphi(t,\cdot;\sigma_bu_{12},g))= N_u$$
 for all $b\in S^1$ and $t\in\mathbb{R}$. Hence,
 \begin{equation}\label{trans-const}
   z(\varphi(t,\cdot;\sigma_au_{2},g)-\varphi(t,\cdot;\sigma_bu_{12},g))=z(\varphi(t,\cdot;u_{2},g)-\varphi(t,\cdot;\sigma_{b-a}u_{12},g))=N_u
 \end{equation}
 for all $a,b\in S^1$ and $t\in\mathbb{R}$.

By virtue of Lemma \ref{hyperbolic1} and \eqref{bnormalnega-proxi2}, there exists some $v_n\in M^{cu}(\varphi(s_n,\cdot;u_2,g),g\cdot s_n,\delta^*)\cap M^s(\varphi(s_n,\cdot;\sigma_{a^*_{12}}u_{12},g),g\cdot s_n,\delta^*)$ for $n\gg 1$. Similarly as the proof in Lemma \ref{inducedproxi}, we can also obtain $v_n\notin M^{c}(\varphi(s_n,\cdot;u_2,g),g\cdot s_n,\delta^*)$. Furthermore, by Lemma \ref{homogeneous}(a) and the foliation statement in Remark \ref{stable-leaf}(4), there is $a_n^*\in S^1$ such that $v_n\in M^u(\sigma_{a_n^*}\varphi(s_n,\cdot;u_2,g),g\cdot s_n,\delta^*)$. So, by Lemma \ref{zerocenter} (2), we have $z(v_n-\sigma_{a_n^*}\varphi(s_n,\cdot;u_2,g))\leq N_u-2$; and hence, Corollary \ref{difference-lapnumber}(a) implies that
 \begin{equation}\label{case-b3-less-1}
   z(\varphi(-s_n,\cdot;v_n,g\cdot s_n)-\sigma_{a_n^*}u_2)\leq N_u-2.
 \end{equation}
On the one hand, by Lemma \ref{zero-cons-local} and the compactness of $S^1$, there exists $\delta>0$ (independent of $a,b\in S^1$) such that for any $v\in X$ with $\|v\|<\delta$, one has
 \begin{equation}\label{case-b3-const-1}
z(\sigma_au_2-\sigma_bu_{12}+v)=N_u.
\end{equation}
On the other hand, since $v_n\in M^s(\sigma_{a^*_{12}}\varphi(s_n,\cdot;u_2,g),g\cdot s_n,\delta^*)$, by Remark \ref{stable-leaf}(1),
\begin{equation*}
\begin{split}
\norm{\varphi(-s_n,\cdot;v_n,g\cdot s_n)-\sigma_{a^{**}_{12}}u_{12}} \leq Ce^{\frac{\alpha}{2} s_n}\|v_n-\varphi(s_n,\cdot;\sigma_{a^{**}_{12}}u_{12},g)\|\le C\delta^*e^{\frac{\alpha}{2} s_n}.\\
\end{split}
\end{equation*}
It entails that $\|\varphi(-s_n,\cdot;v_n,g\cdot s_n)-\sigma_{a^{**}_{12}}u_{12}\|<\delta$ for $n$ sufficiently large. As a consequence, \eqref{case-b3-const-1} implies that $z(\varphi(-s_n,\cdot;v_n,g\cdot s_n)-\sigma_{a_n^*}u_2)=z(\varphi(-s_n,\cdot;v_n,g\cdot s_n)-\sigma_{a^{**}_{12}}u_{12}+\sigma_{a^{**}_{12}}u_{12}-\sigma_{a_n^*}u_2)=z(\sigma_{a^{**}_{12}}u_{12}-\sigma_{a_n^*}u_2)$. So, by \eqref{trans-const}, one obtain that $z(\varphi(-s_n,\cdot;v_n,g\cdot s_n)-\sigma_{a_n^*}u_2)=N_u$, a contradiction to \eqref{case-b3-less-1}.
\end{proof}

In summary, we have deduced certain contradiction in each of the above three sub-lemmas, which enables us to complete the proof of the fact that case-(iii) in Lemma \ref{gen-hyper2} can not happen. In other words,  we have proved the first statement of Theorem \ref{norma-hyper}(i), that is, there is a minimal set $M\subset \O$ of $\Pi^t$ such that $\O\subset \Sigma M$.

\vskip 3mm
As for the proof of the remaining part in Theorem \ref{norma-hyper}(i), we first claim that $M$ here is spatially-inhomogeneous (Otherwise, $\O\subset \Sigma M=M$, which implies that $\O=M$. Hence, by {\bf (H1)}, Lemma \ref{hyperbolic-minimal}(2) entails that ${\rm dim}V^u(\O)={\rm dim}V^u(M)=0$, a contradiction). So, we can repeat the same argument from the third paragraph of the proof Theorem \ref{structure-thm} to the end of Theorem \ref{structure-thm}. As a matter of fact, one can even obtain that $\hat{M}$ is a $1$-cover, because $\tilde{M}$ is a $1$-cover in Theorem \ref{norma-hyper}.
So,  for each $g\in H(f)$ (instead of just $g\in H_0(f)$), we can obtain all the statements from \eqref{E:u-g--base-trn1}-\eqref{circle-flow-eq41}. In particular, $u_{g\cdot t}(x)$ in \eqref{E:rotation-spiral1} is almost-periodic in $t$ uniformly in $x$; and the function $\dot {c}^g(t)=G(t,c^g(t))$ is time almost-periodic if $f$ is uniformly almost periodic in $t$. Thus, we have naturally induces
an almost-periodically forced skew-product flow on $\mathcal{S}^1\times H(f)$, which completes the proof of Theorem \ref{norma-hyper}.
\end{proof}


\subsection{Proof of Theorem \ref{hyperbolic0}}\label{Proof-Thm5.3}
In this subsection, we will prove Theorem \ref{hyperbolic0}. We point out that,
  if $\dim V^u(\Omega)=0$ in Theorem \ref{hyperbolic0}, then $\Omega$ is uniformly stable because {\bf (H0)} holds. Then it follows from \cite[Theorem II.2.8]{Shen1998} that $\Omega$ is a uniformly stable minimal set. Moreover, by \cite[Theorem 4.1]{SWZ}, $\Omega$ is spatially-homogeneous minimal set and $1$-cover of the base $H(f)$. Thus, in the remaining part of this section we always assume that ``{\it {\bf (H0)} holds with $\dim V^u(\Omega)>0$."}

\begin{proof}[Proof of Theorem \ref{hyperbolic0}]
By Remark \ref{invari-space}, any minimal set $M\subset\O$ is hyperbolic. So, Lemma \ref{hyperbolic-minimal}(1) entails that $M$ is a spatially-homogeneous $1$-cover. In particular, $\Sigma M=M$.  By virtue of the statement in the beginning of proof of Theorem \ref{structure-thm}, one knows that Theorem \ref{structure-thm} still holds under {\bf (H0)}. As a consequence, one of the following must hold for $\Omega$:
\begin{itemize}
\item[{ \rm (i)}] $\Omega$ is a minimal invariant set.
\item[{ \rm (ii)}] $\Omega=M_1\cup M_{11}$, where $M_1$ is minimal, $M_{11}\neq \emptyset$, $M_{11}$ connects $M_1$ in the sense that if $(u_{11},g)\in M_{11}$, then $M_1\subset\omega(u_{11},g)\cap\alpha (u_{11},g)$.
\item[{ \rm (iii)}] $\Omega=M_1\cup M_2\cup M_{12}$, where $M_1$, $M_2$ are minimal sets, $M_{12}\not =\emptyset$, and for any $u_{12}\in M_{12}$, either  $M_1\subset \omega(u_{12},g)$ and $M_2\cap \omega(u_{12},g)=\emptyset$, or $M_2\subset \omega(u_{12},g)$ and $M_1\cap \omega(u_{12},g)=\emptyset$, or $M_1\cup M_2\subset \omega(u_{12},g)$ (and analogous for $\alpha(u_{12},g)$).
\end{itemize}

We only need to prove that neither (ii) nor (iii) can occur. In fact, when (ii) holds, let $\{(u_1,g)\}=M_1\cap p^{-1}(g)$. Choose any $(u_{11},g)\in M_{11}$. It then turns out that $\{(u_1,g),(u_{11},g)\}$ is a two sided proximal pair, which contradicts to Lemma \ref{hyperbolic2}.

When (iii) holds, then $\Omega=M_1\cup M_2\cup M_{12}$.  Let $\{(u_i,g)\}=M_i\cap p^{-1}(g)$ for $i=1,2$ and any $g\in H(f)$.
Given any $(u_{12},g)\in M_{12}$, Lemma \ref{hyperbolic2} implies that neither $\{(u_1,g),(u_{12},g)\}$ nor $\{(u_1,g),(u_{12},g)\}$ forms a two sided proximal pair. Therefore, without loss of generality, we may assume that $\omega(u_{12},g)\cap M_1\neq \emptyset$, $\alpha(u_{12},g)\cap M_2\neq \emptyset$. Consequently, it is easy to see that $\Pi^{t}(u_{12},g)-\Pi^{t}(u_{1},g)\to 0$ (resp. $\Pi^{t}(u_{12},g)-\Pi^{t}(u_{2},g)\to 0$) as $t\to\infty$ (resp. $t\to -\infty$). By Remark \ref{stable-leaf}(1), we have $\varphi(t,\cdot;u_{12},g)\in M^s((\Pi^t(u_{1},g),\delta^*)$ (resp. $\varphi(t,\cdot;u_{12},g)\in M^u((\Pi^t(u_{2},g),\delta^*)$) for $t\gg 1$ (resp. $t\ll-1$). Since {\bf (H0)} holds and $\dim V^u(\Omega)>0$, it follows from Lemma \ref{hyperbolic2} that $\dim V^u(\Omega)$ should be odd. As a consequence, by Lemma \ref{zerocenter}(1), one has
 \begin{equation*}
   z(\varphi(t,\cdot;u_{12},g)-\varphi(t,\cdot;u_1,g))\geq N_u\ge 2, \quad t\gg 1.
 \end{equation*}
Together with Corollary \ref{difference-lapnumber}(a), this implies that
\begin{equation}\label{positive-zero-estima}
   z(\varphi(t,\cdot;u_{12},g)-\varphi(t,\cdot;u_1,g))\geq N_u, \quad \forall t\in \mathbb{R}^1.
 \end{equation}
Noticing that both $M_1$ and $M_2$ are spatially-homogenous, it is easy to see that $z(\varphi(t,\cdot;u_1,g)-\varphi(t,\cdot;u_2,g))=0$ for any $t\in \mathbb{R}$. However, let $t_n\to -\infty$ be such that $\Pi^{t_n}(u_{12},g)\to (u_2,g)$ and $\Pi^{t_n}(u_{1},g)\to(u_1,g)$ as $n\to\infty$. Then Lemma \ref{zero-cons-local} implies that there is $N\in\mathbb{N}$ such that, for any $n>N$, one has
 \begin{equation}
   z(u_1-u_2)=z(\varphi(t_n,\cdot;u_{12},g)-\varphi(t_n,\cdot;u_1,g)).
 \end{equation}
So, by \eqref{positive-zero-estima}, $z(u_1-u_2)\geq N_u\geq 2$, a contradiction. Thus, we have completed the proof of this theorem.
\end{proof}


\section{Appendix}

In this Appendix, we will present an example to illustrate that, for the time almost-periodic cases, one can not expect that any omega-limit set is imbedded into an almost periodically forced circle flow on $S^1$.
Compared with the time periodic cases discussed in \cite[Theorem 1]{SF1}, this reveals that there are essential differences between time-periodic cases and non-periodic cases.

\vskip 2mm
Consider the following parabolic equation:
\begin{equation}\label{examp1}
  u_t=u_{xx}+u_x+(f(t)+1)u,\,\,t>0,\,x\in S^{1}=\mathbb{R}/2\pi \mathbb{Z},
\end{equation}
where $f(t)=-\sum_{k=1}^{\infty}2^{-k}\pi \mathrm{sin}(2^{-k}\pi t)$ is an almost periodic function.

The skew-product semiflow $\Pi^t$ on $X\times H(f)$ is
\begin{equation}
  \Pi^t(u,g)=(\varphi(t,\cdot; u,g),g\cdot t),
\end{equation}
where $X$ is the fractional power space defined in the introduction. Let $u_0=\sin x$, then $\varphi(t,\cdot; u_0,f)=e^{\int_{0}^t f(s) ds}\sin(x+t)$ is the solution of \eqref{examp1} with the initial value $\varphi(0,\cdot; u_0,f)=u_0$.

 Following the discussion in \cite{Sacker77,ShenYi-2}, the function $\phi(t)=e^{\int_{0}^t f(s) ds}$ satisfies the following properties:

(a) $\phi(t)$ is bounded for $t\geq 0$;

(b) There exists $t_n\to\infty$ such that $\phi(t_n)\to 0$ as $n\to\infty$, and $\phi(2^n)\geq e^{-2\pi-2}$ for $n=1,2,\cdots$;

(c) For any sequence $t_n\to\infty$ such that $\lim_{n\to\infty}\phi(t+t_n)=\phi^*(t)$ exists, $\phi^*(t)$ is not almost periodic if it is nonzero.

By virtue of (a)-(c), the $\omega$-limit set $\omega(u_0,f)$ is not minimal, and $M=\{0\}\times H(f)$ is the unique minimal set contained in $\omega(u_0,f)$. Moreover, $\omega(u_0,f)$ is an almost $1$-cover of $H(f)$ (see, e.g. the similar argument in \cite[p.396]{ShenYi-2}).

Let $\omega(\phi(0),c(0),f)$ be the $\omega$-limit set of the flow $\{(\phi(t),c(t),f\cdot t)\subset \mathbb{R}\times S^1\times H(f): t\in \mathbb{R}\}$, where the function $t\mapsto c(t):=t\text{ (mod }2\pi)\in S^1$. Then, for any $(u,g)\in \omega(u_0,f)$, one has $u=\phi^*_g\sin(x+c^*_g)$ with $(\phi^*_g,c^*_g,g)\in \omega(\phi(0),c(0),f)$. Therefore, whenever $(u,g)\in \omega(u_0,f)\setminus M=\omega(u_0,f)\setminus (\Sigma M)$, we have $\phi^*_g\neq 0$; and hence, $u=\phi^*_g\sin(x+c^*_g)$ is spatially-inhomogeneous. Moreover, let $H_1(f):=\{g\in H(f):\text{there exists some } (u,g)\in \omega(u_0,f)\setminus M\}$. Then, for any $g\in H_1(f)$, there does not exist $u_g\in X$ such that $\omega(u_0,f)\cap p^{-1}(g)\subset (\Sigma u_g,g)$, where $\Sigma u_g$ is the $S^1$-group orbit of $u_g$ defined in \eqref{E:group-orbit-11}.
As a consequence, we have:

\vskip 3mm
\noindent $\bullet$ {\it $\omega(u_0,f)$ cannot be imbedded into an almost periodically forced circle flow on $S^1$.}
\vskip 3mm

Moreover, we have the some further observations:\vskip 2mm

\noindent {\bf Remark A.1.} (i) The Sacker-Sell spectrum of $\omega(u_0,f)$ is $\sigma(\omega(u_0,f))=\{1,0,\cdots,1-k^2,\cdots\}$ with $\dim V^c(\omega(u_0,f))=2$ and $\dim V^u(\omega(u_0,f))=1$.

(ii) $\omega(u_0,f)$ is neither spatially-homogeneous nor spatially-inhomogeneous.

(iii) This example also reveals that, even if $\dim V^c(\O)=2$ and $\dim V^u(\Omega)$ is odd,
 $\O\subset \Sigma M$ (see Theorem \ref{structure-thm}(i)) does not always hold. As a matter of fact, this example satisfies Theorem \ref{structure-thm}(ii).

\end{document}